\documentclass[graybox]{svmult}
\RequirePackage[numbers]{natbib}
\usepackage{pictexwd}
\usepackage{amssymb,graphicx,url}
\usepackage[colorlinks,citecolor=blue,urlcolor=blue]{hyperref}
\usepackage{bm}
\usepackage{lipsum}
\usepackage{color}

\def\a{\alpha}
\def\b{\beta}

\def\R{\mathbb R}

\def\d{\delta}

\def\B{{\cal B}}

\def\E{{\mathbb E}}

\def\P{{\mathbb P}}

\def\labda1{\lambda_1}
\def\labda2{\lambda_2}
\def\m{\mu}

\def\e{\varepsilon}
\def\f{\phi}
\def\t{\tau}

\def\s{\sigma}

\def\comment#1{\relax}

\def\=in{\mathop{\rm =}}

\usepackage{type1cm}        % activate if the above 3 fonts are
\usepackage{makeidx}         % allows index generation
\usepackage{graphicx}        % standard LaTeX graphics tool
\usepackage{multicol}        % used for the two-column index
\usepackage[bottom]{footmisc}% places footnotes at page bottom
 \usepackage{amsmath}

\usepackage{newtxtext}       % 
\usepackage{newtxmath}       % selects Times Roman as basic font

\usepackage{amsmath,here,amsfonts,latexsym,graphicx,here}
\usepackage{graphicx}
\usepackage{caption}
\usepackage{subcaption}
\usepackage{float}

\def\a{\gamma}
\def\b{\beta}
\def\d{\delta}
\def\t{\tau}
\def\m{\mu}

\def\bma{\bm\alpha}
\def\bmb{\bm\beta}

\def\cS{\mathcal{S}}
\def\P{{\mathbb P}}

\def\bmA{\bm A}
\def\bmX{\bm X}
\def\bmx{\bm x}

\def\a{\alpha}
\def\m{\mu}

\def\X{\mathcal{X}}

\definecolor{purple}{rgb}{0.55,0.2,0.90}
\usepackage{bm}

\makeindex

\begin{document}

\title*{Profile least squares estimators in the monotone single index model}

\titlerunning{LS estimators in the monotone single index model}

\author{Fadoua Balabdaoui and Piet Groeneboom}
% Use \authorrunning{Short Title} for an abbreviated version of
% your contribution title if the original one is too long
\institute{Fadoua Balabdaoui \at Seminar f\"ur Statistik ETH, Z\"urich, Switzerland,\\ \email{fadoua.balabdaoui@stat.math.ethz.ch},
\and Piet Groeneboom \at Delft University of Technology, DIAM, Mekelweg 4, 2628 CD Delft, The Netherlands,\\
\email{P.Groeneboom@tudelft.nl}}

\maketitle

\abstract*{}

\abstract{We consider least squares estimators of the finite regression parameter $\bma$ in the single index regression model $Y=\psi(\bma^T\bmX)+\e$, where $\bmX$ is a $d$-dimensional random vector, $\E(Y|\bmX)=\psi(\bma^T\bmX)$, and where $\psi$ is monotone.  It has been suggested to estimate $\bma$ by a profile least squares estimator, minimizing $\sum_{i=1}^n(Y_i-\psi(\bma^T\bmX_i))^2$ over monotone $\psi$ and $\bma$ on the boundary $\cS_{d-1}$of the unit ball. Although this suggestion has been around for a long time, it is still unknown whether the estimate is $\sqrt{n}$ convergent.
We show that a profile least squares estimator, using the same pointwise least squares estimator for fixed $\bma$, but using a different global sum of squares, is $\sqrt{n}$-convergent and asymptotically normal. The difference between the corresponding loss functions is studied and also a comparison with other methods is given.}

\section{Introduction}
\label{section:intro}
\setcounter{equation}{0}
The monotone single index model tries to predict a response from the linear combination of a finite number of parameters and a function linking this linear combination to the response via a monotone {\it link function} $\psi_0$ which is  unknown. So, more formally, we have the model
\begin{align*}
Y=\psi_0(\bma_0^T\bmX)+\e,
\end{align*}
where $Y $ is a one-dimensional random variable, $\bm X = (X_1,\ldots, X_d)^T$ is a $d$-dimensional random vector with distribution function $G$,  $\psi_0$ is monotone  and $\e$ is a one-dimensional random variable such that $\E[\e | \bm X] = 0$ $G$-almost surely. For identifiability, the regression parameter $\bma_0$ is a vector  of norm $\|\bma_0\|_2=1$, where $\|\,  \cdot \|_2$ denotes the Euclidean norm in $\R^d$, so $\bma_0\in\cS_{d-1}$, the unit $(d-1)$-dimensional sphere.

The ordinary profile least squares estimate of $\bma_0$ is an $M$-estimate in two senses: for fixed $\bma$ the least squares criterion
\begin{align}
\label{step1}
\psi\mapsto n^{-1}\sum_{i=1}^n\left\{Y_i-\psi(\bma^T\bmX_i)\right\}^2
\end{align}
is minimized for all monotone functions $\psi$ (either decreasing or increasing) which gives an $\bma$ dependent function $\hat\psi_{n,\bma}$, and the function
\begin{align}
\label{step2}
\bma\mapsto n^{-1}\sum_{i=1}^n\left\{Y_i-\hat\psi_{n,\bma}(\bma^T\bmX_i)\right\}^2
\end{align}
is then minimized over $\bma$. This gives a profile least squares estimator $\hat\bma_n$ of $\bma_0$, which we will call LSE in the sequel.
Although this estimate of $\bma_0$ has been  known now for a very long time (more than 30 years probably), it is not known whether it is $\sqrt{n}$ convergent (under appropriate regularity conditions), let alone that we know its asymptotic distribution. Also, simulation studies are rather inconclusive. For example, it is conjectured in \cite{tanaka2008} on the basis of simulations that the rate of convergence of $\hat\bma_n$ is $n^{9/20}$. Other simulation studies, presented in \cite{BDJ:19}, are also inconclusive. In that paper, it was also proved that an ordinary least squares estimator (which ignores that the link function could be non-linear) is $\sqrt{n}$-convergent and asymptotically normal under elliptic symmetry of the distribution of the covariate $\bmX$. Another linear least squares estimator of this type, where the restriction on $\bma$ is $\bma^T\bm S_n\bma=1$, where $S_n$ is the usual estimate of the covariance matrix of the covariates, and where a renormalization at the end is not needed (as it is in the just mentioned linear least squares estimator) was studied in \cite{FGH:19} and there shown to have similar behavior.
If this suggests that the profile LSE should also be $\sqrt n$-consistent,  the extended simulation study in \cite{FGH:19} shows that it is possible to find other estimates which exhibit a better performance in these circumstances.

%\cite{FGH:19} \comF{are also inconclusive}. \comF{In \cite{FGH:19}, it was also proved that the ordinary least squares estimator (which ignores that the link function could be non-linear) is $\sqrt{n}$-convergent and asymptotically normal under elliptic symmetry of the distribution of the covariate $\bmX$.  

An alternative way to estimate the regression vector is to minimize the criterion
\begin{align}
\label{SSE_criterion}
\bma\mapsto \left\|n^{-1}\sum_{i=1}^n\left\{Y_i-\hat\psi_{n,\bma}(\bma^T\bmX_i)\right\}\bmX_i\right\|^2
\end{align}
over $\bma\in \cS_{d-1}$, where $\|\cdot\|$ is the Euclidean norm. Note that this is the sum of $d$ squares.  The rational behind minimizing (\ref{SSE_criterion}) is the fact that the true index vector, $\bma_0$, satisfies the (population) score equation
\begin{eqnarray}\label{ps}
\mathbb E \left\{ (Y - \psi_0(\bma_0^T \bmX)) \bmX  \theta(\bma_0^T \bmX)\right \}  =  \mathbf{0}.
\end{eqnarray}
where $\theta$ is any measurable and bounded function.  This clearly follows from the iterative law of expectations and the fact that $\mathbb E\{ Y | \bma_0^T \bmX \}  =  \psi_0(\bma_0^T \bmX)$.  If the function $\theta $ is taken to be the constant $1$, then the goal is to find the minimizer of the Euclidean norm of the empirical counterpart of the above score equation, after replacing the unknown link function, $\psi_0$, by its estimator $\hat \psi_{n, \bma}$. 

We prove in Section \ref{section:analysis_SSE} that this minimization procedure  leads to a $\sqrt{n}$ consistent and asymptotically normal estimator, which is a more precise and informative result compared to what we know now about  the LSE..  Using the well-known properties of isotonic estimators, it is easily seen that the function (\ref{SSE_criterion}) is piecewise constant as a function of $\bma$, with finitely many values, so the minimum exists and is equal to the infimum over $\bma\in \cS_{d-1}$. Notice that this estimator does not use any tuning parameters, just like the LSE.

In \cite{FGH:19}, a similar Simple Score Estimator (SSE) $\hat\bma_n$ was defined as a point $\bma\in\cS_{d-1}$ where all components of the function
\begin{align*}
\bma\mapsto n^{-1}\sum_{i=1}^n\left\{Y_i-\hat\psi_{n,\bma}(\bma^T\bmX_i)\right\}\bmX_i
\end{align*}
cross zero. If the criterion function were continuous in $\bma$, this estimator would have been the same as the least squares estimator, minimizing (\ref{SSE_criterion}), with a minimum equal to zero, but in the present case we cannot assume this because of the discontinuities of the criterion function.

The definition of an estimator as a crossing of the $d$-dimensional vector $\bm0$  makes it necessary to prove the existence of such an estimator, which we found to be a rather non-trivial task. Defining our estimator directly as the minimizer of (\ref{SSE_criterion}), so as a least squares estimator,  relieves us from the duty to prove its existence. Since our estimator has the same limit distribution as  the SSE, we refer to it here under the same name.

A fundamental function in our treatment is the function $\psi_{\bma}$, defined as follows.

\begin{definition}
\label{def_psi_alpha}
{\rm
Let $\cS_{d-1}$ denote again the boundary of the unit ball in $\R^d$. Then, for each $\bma\in \cS_{d-1}$, the function $\psi_{\bma}:\R\to\R$ is defined as the nondecreasing function which minimizes
\begin{align*}
\psi\mapsto\E\{Y-\psi(\bma^T\bmX)\}^2
\end{align*}
over all nondecreasing functions $\psi:\R\to\R$. The existence and uniqueness of the function $\psi_{\bma}$ follows for example from the results in \cite{landers_rogge:81}.

The function $\psi_{\bma}$ coincides in a neighborhood of $\bma_0$ with the ordinary conditional expectation function $\tilde\psi_{\bma}$
\begin{align}
\label{conditional_exp}
\tilde\psi_{\bma}(u)=\E\left\{\psi_0(\bma_0^T\bmX)|\bma^T\bmX=u\right\},\qquad u\in\R,
\end{align}
see \cite{FGH:19}, Proposition 1. The general definition of $\psi_{\bma}$ uses conditioning on a $\sigma$-lattice, and $\psi_{\bma}$ is also called a {\it conditional $2$-mean} (see \cite{landers_rogge:81}).
}
\end{definition}

\vspace{0.3cm}
The importance of the function $\psi_{\bma}$ arises from the fact that we can differentiate this function w.r.t.\ $\bma$, in contrast with the least squares estimate $\hat\psi_{n,\bma}$, and that $\psi_{\bma}$ represents the least squares estimate of $\psi_0$ in the underlying model for fixed $\bma$, if we use $\bma^T\bmx$ as the argument of the monotone link function.

It is also possible to introduce a tuning parameter and use an estimate of $\frac{d}{du}\psi_{\bma}(u)\bigr|_{u=\bma^T\bmX}$. This estimate is defined by:
\begin{align}
\label{estimate_derivative_psi}
\tilde\psi_{n,h,\bm\a}'(u)=\frac1h\int K\left(\frac{u-x}h\right)\,d \hat\psi_{n,\bm\a}(x),
\end{align}
where $K$ is one of the usual kernels, symmetric around zero and with support $[-1,1]$, and where $h$ is a bandwidth of order $n^{-1/7}$ for sample size $n$. For fixed $\bma$, the least squares estimate $\hat\psi_{n,\bma}$ is defined in the same way as above. 
Note that this estimate is rather different from the derivative of a Nadaraya-Watson estimate which is also used in this context and which is in fact the derivative of a ratio of two kernel estimates. If we use the Nadaraya-Watson estimate we need in principle two tuning parameters, one for the estimation of $\psi_0$ and another one for the estimation of the derivative $\psi'_0$.

Using the estimate (\ref{estimate_derivative_psi}) of the derivative we now minimize
\begin{align}
\label{ESE_criterion}
\bma\mapsto \left\|n^{-1}\sum_{i=1}^n\left\{Y_i-\hat\psi_{n,\bma}(\bma^T\bmX_i)\right\}\bmX_i\,\tilde\psi_{n,h,\bm\a}'(\bma^T\bmX_i)\right\|^2
\end{align}
instead of (\ref{SSE_criterion}), where $\|\cdot\|$ is again the Euclidean norm.  The motivation for considering such a minimization problem is very similar to the one given above for the SSE. The only difference now is that the current approach allows us to take the function $\theta$ to be equal to the derivative of $\psi'_0$, which is replaced in the empirical version of the population score in (\ref{ps}) by its estimator $\tilde\psi_{n,h,\bm\a}' $.  A variant of this estimator was defined in \cite{FGH:19} and called the Efficient Score Estimator (ESE) there, since, if the conditional variance $\text{var}(Y|\bmX=\bmx)=\s^2$, where $\s^2$ is independent of the covariate $\bmX$  (the homoscedastic model), the estimate is efficient. As in the case of the simple score estimator (SSE), the estimate was defined as a crossing of zero estimate in \cite{FGH:19} and not as a minimizer of (\ref{ESE_criterion}). But the definition as a minimizer of (\ref{ESE_criterion}) produces an estimator that has the same limit distribution. For reasons of space, we will only give a sketch of the proof of this statement below in Section \ref{sec:ESE_PLSE}.

The qualification ``efficient'' is somewhat dubious, since the estimator is no longer efficient if we do not have homoscedasticy. We give an example of that situation in Section \ref{section:simulations}, where, in fact, the SSE has a smaller asymptotic variance than the ESE.
Nevertheless, to be consistent with our treatment in \cite{FGH:19} we will call the estimate, $\hat\bma_n$, minimizing (\ref{ESE_criterion}), again the ESE. 

Dropping the monotonicity constraint, we can also use as our estimator of the link function a cubic spline $\hat\psi_{n,\bma}$, which is defined as the function minimizing
\begin{align}
\label{spline}
\sum_{i=1}^n\left\{\psi(\bma^T\bmX_i)-Y_i\right\}^2+\mu\int_a^b \psi''(x)^2\,dx,
\end{align}
over the class of functions ${\cal S}_2[a,b]$ of differentiable functions $\psi$ with an absolutely continuous first derivative, where
\begin{align*}
a=\min_i\bma^T\bmX_i,\qquad b=\max_i\bma^T\bmX_i,
\end{align*}
see \cite{green_silverman:94}, pp. 18 and 19, where $\m>0$ is the penalty parameter. Using these estimators of the link function, the estimate $\hat\bma_n$ of $\bma_0$ is then found in \cite{arun_rohit:20} by using a $(d-1)$-dimensional parameterization $\b$ and a transformation $S:\bm\beta\mapsto S(\bmb)=\bma$, where $S(\bmb)$ belongs to the surface of the unit sphere in $\R^d$, and minimizing the criterion
\begin{align*}
\bmb\mapsto\sum_{i=1}^n\{Y_i-\hat\psi_{S(\bmb),\m}(S(\bmb)^T\bmX_i)\}^2,
\end{align*}
over $\bmb$, where $\hat\psi_{S(\bmb),\m}$ minimizes (\ref{spline}) for fixed $\bma=S(\bmb)$.

Analogously to our approach above we can skip the reparameterization, and minimize instead:
\begin{align}
\label{spline_eq}
\left\|\frac1n\sum_{i=1}^n \bigl\{\hat\psi_{n,\bma,\m}(\bma^T\bmX_i)-Y_i\bigr\}\bmX_i\,\tilde\psi_{n,\bm\a,\m}'(u)\bigr|_{u=\bma^T\bmX_i}\right\|
\end{align}
where $\tilde\psi_{n,\bma,\m}$ minimizes (\ref{spline}) for fixed $\bma$ and $\tilde\psi'_{n,\bma,\m}$ is its derivative. We call this estimator the spline estimator.

We finally give simulation results for these different methods in Section \ref{section:simulations}, where, apart from the comparison with the spline estimator, we make a comparison with other estimators of $\bma_0$ not using the monotonicity constraint: the Effective Dimension Reduction (EDR) method, proposed in \cite{hristache01} and implemented in the {\tt R} package {\tt edr}, the (refined) MAVE (Mean Average conditional Variance Estimator) method, discussed in \cite{xia:06}, and implemented in the {\tt R} package {\tt MAVE}, and EFM (Estimation Function Method), discussed in \cite{cui2011}.

For reasons of space, the proofs of the statements of our paper are given in \cite{Fadoua_piet:20}.

\section{General conditions and the functions $\hat\psi_{n,\hat\bma}$  and $\psi_{\hat\bma}$}
\label{section:psi_alpha}
\setcounter{equation}{0}
We  give general conditions that we assume to hold in the remainder of the paper here and give graphical comparisons of the functions $\hat\psi_{n,\bma}$  and $\psi_{\bma}$, where $\psi_{\bma}$ is defined in Definition \ref{def_psi_alpha}.

\vspace{0.3cm}
\begin{example}
\label{example_uniform_covariates}
{\rm
As an illustrative example we take $d=2$, $\psi_0(x)=x^3$, $\bma_0=(1/\sqrt{2},1/\sqrt{2})^T$, $Y_i=\psi_0(\bma_0^TX_i)+\e_i$, where the $\e_i$ are i.i.d.\ standard normal random variables, independent of the $\bmX_i$, which are i.i.d.\ random vectors, consisting of two independent Uniform$(0,1)$ random variables.
In this case the conditional expectation function (\ref{conditional_exp}) is a rather complicated function of $\bma$ which we shall not give here, but can be computed by a computer package such as Mathematica or Maple.
The loss functions:
\begin{align}
\label{loss_functions}
L^{\text{LSE}}:\a_1\mapsto\E\{Y-\psi_{\bma}(\bma^T\bmX)\}^2\qquad\text{and}\qquad
\widehat L^{\text{LSE}}_n:\a_1\mapsto n^{-1}\sum_{i=1}^n\bigl\{Y_i-\hat\psi_{n,\bma}(\bma^T\bmX_i)\bigr\}^2
\end{align}
where the loss function $\widehat L^{\text{LSE}}_n$ is for sample sizes $n=10,000$ and $n=100,000$, and $\bma=(\a_1,\a_2)^T$. For $\a_1\in[0,1]$ and $\a_2$ equal to the positive root $\{1-\a_1^2\}^{1/2}$, we get Figure \ref{fig:empirical_vs_analytic}.
The function $L^{\text{LSE}}$ has a minimum equal to $1$ at $\a_1=1/\sqrt{2}$ and $\widehat L^{\text{LSE}}_n$ has minimum at a value very close to $1/\sqrt{2}$ (furnishing the profile LSE $\hat\bma_n$), which gives a visual evidence for consistency of the profile LSE.
}
\end{example}

\begin{figure}[!h]
	\centering
	\begin{subfigure}{0.45\linewidth}
	\includegraphics[width=0.95\textwidth]{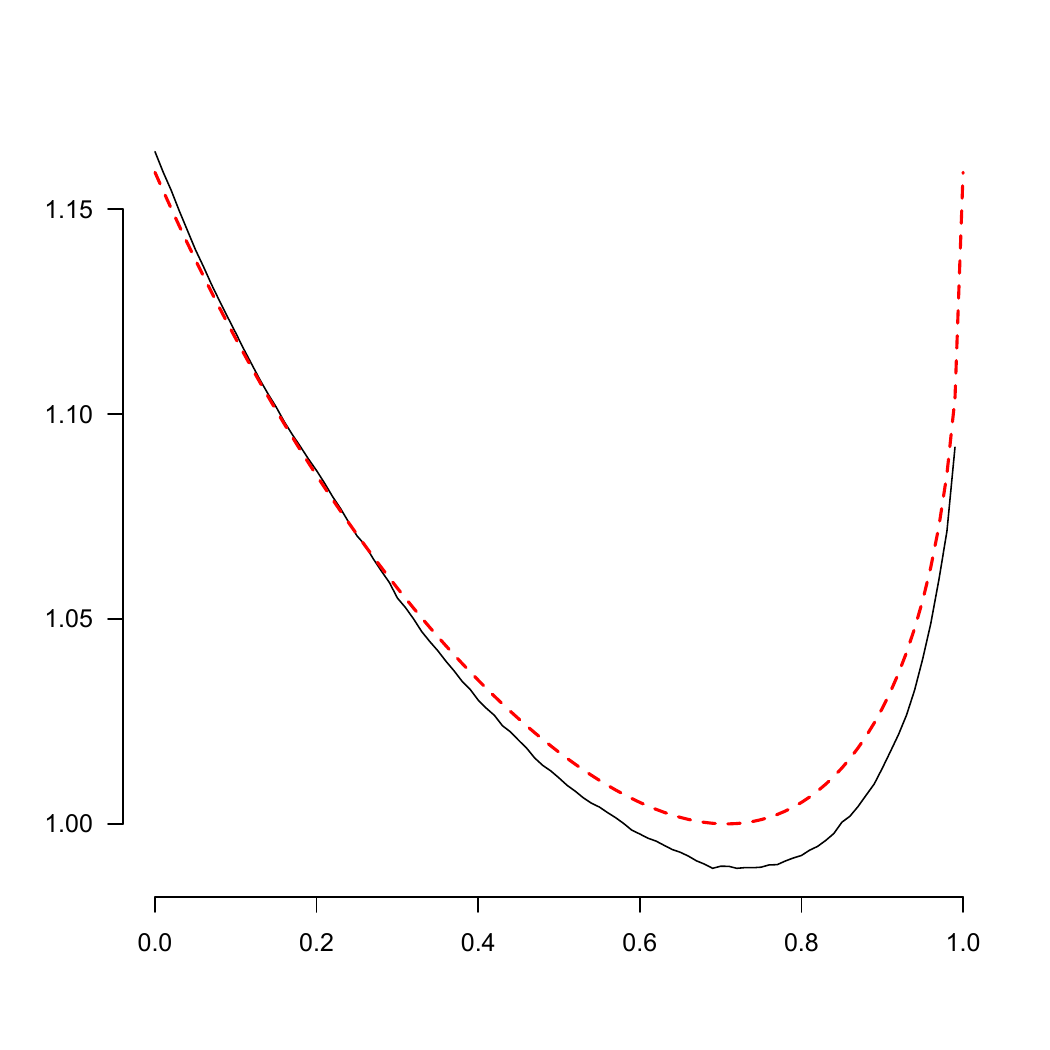}
	\caption{$n = 10^4$}
	\end{subfigure}
	\begin{subfigure}{0.45\linewidth}
	\includegraphics[width=0.95\textwidth]{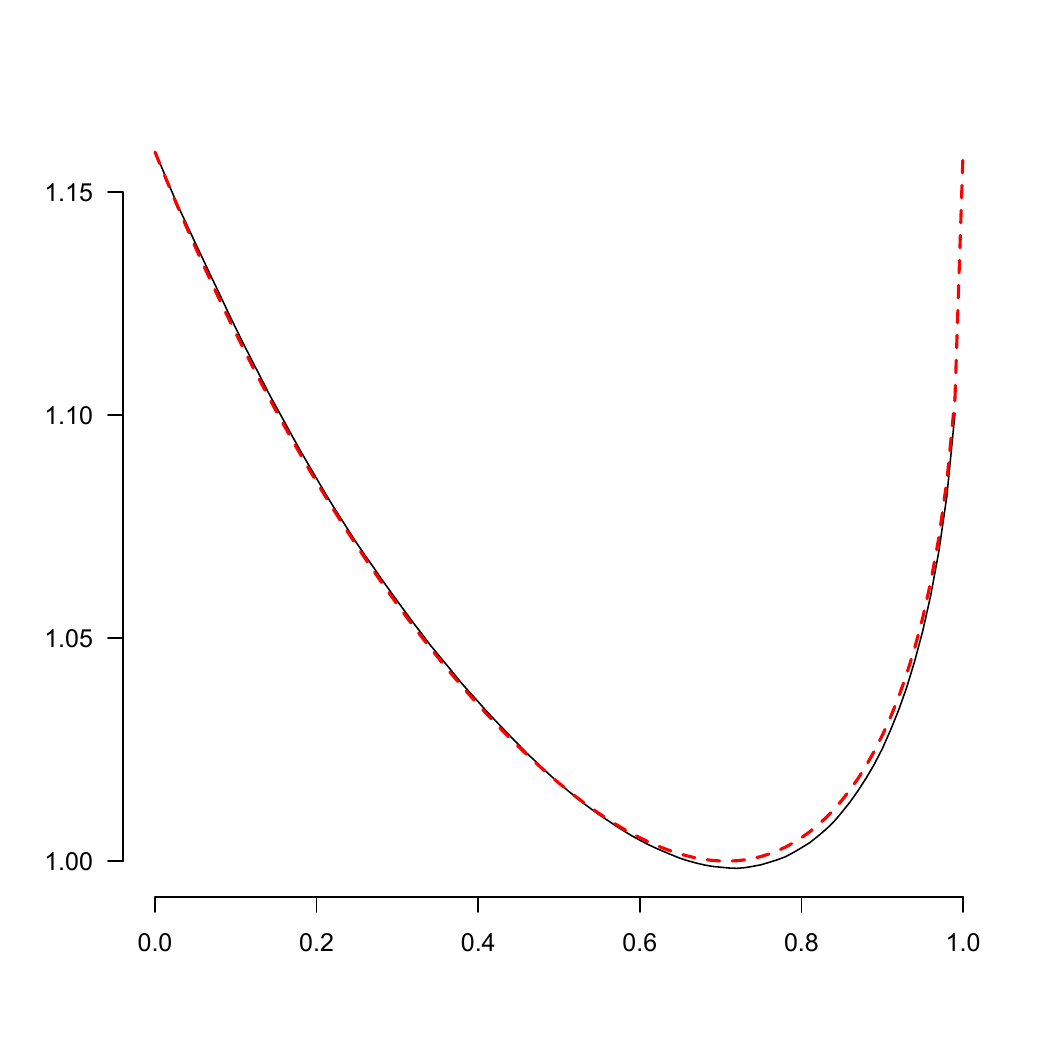}
	\caption{$n =10^5$}
\end{subfigure}\\
\caption{The loss functions $L^{\text{LSE}}$ (red, dashed) and $\widehat L^{\text{LSE}}_n$ (solid), where $n=10^4$ and $n=10^5$.}
\label{fig:empirical_vs_analytic}
	\end{figure}

In order to show the $\sqrt{n}$- consistency and asymptotic normality of the estimators in the next sections, we now introduce some conditions, which correspond to those in \cite{FGH:19}. We note that we do not need conditions on 
re parameterization. 

\begin{enumerate}
\item[(A1)] $\bmX$ has a density w.r.t.\ Lebesgue measure on its support $\X$, which is a convex  set $\mathcal{X}$ with a nonempty interior,
and satisfies $\mathcal{X}\subset\{\bmx\in\R^d:\|\bmx\|\le R\}$ for some $R>0$.
\item[(A2)] The function $\psi_0$ is bounded on the set $\{u\in\R:u=\bma_0^T\bmx,\,\bmx\in \mathcal{X}\}$.
\item[(A3)] There exists $\d>0$ such that the conditional expectation $\tilde\psi_{\bma}$, defined by (\ref{conditional_exp}) is nondecreasing on $I_{\bma}=\{u\in\R: u=\bma^T\bmx,\,x\in\mathcal{X}\}$ and satisfies $\tilde\psi_{\bma}=\psi_{\bma}$, so minimizes
\begin{align*}
\left\|\E\left\{Y-\psi(\bma^T\bmX)\right\}\bmX\right\|^2,
\end{align*}
over nondecreasing functions $\psi$, if $\|\bma-\bma_0\|\le\d$.
\item[(A4)] Let $a_0$ and $b_0$ be the (finite) infimum and supremum of the interval $\{\bma_0^T\bmx,\,\bmx\in\mathcal{X}\}$. Then $\psi_0$ is continuously differentiable on $(a_0-\d R,a_0+\d R)$, where $R$ and $\d$ are as in Assumption A1 and A3.
\item[(A5)] The density $g$ of $\bmX$ is differentiable and there exist strictly positive constants $c_1$ to $c_4$ such that $c_1\le g(\bmx)\le c_2$ and $c_3\le\frac{\partial}{\partial x_i}g(\bmx)\le c_4$ for $\bmx$ in the interior of $\X$.
\item[(A6)] There exists a $c_0>0$ and $M>0$ such that $\E\{|Y|^m|\bmX=\bmx\}\le m!M_0^{m-2}c_0$ for all integers $m\ge2$ and $\bmx\in\X$ almost surely w.r.t.\ $dG$.
\end{enumerate}

These conditions are rather natural, and are discussed in \cite{FGH:19}. The following lemma shows that, for the asymptotic distribution of $\hat\bma_n$, we can reduce the derivation to the analysis of $\psi_{\hat\bma_n}$. We have the following result (Proposition 4 in \cite{FGH:19}) on the distance between $\hat\psi_{n,\hat\bma}$  and $\psi_{\hat\bma}$.

\begin{lemma}
\label{lemma:closeness_psi_n_alpha_psi_alpha}
Let conditions (A1) to (A6) be satisfied and let $G$ be the distribution function of $\bmX$. Then we have, for $\bma$ in a neighborhood $\B(\bma_0,\d)$ of $\bma_0$:
\begin{align*}
	\sup_{\bma \in \B(\bma_0,\d)} \int\left\{\hat \psi_{n\bma}(\bma^T\bm x)-\psi_{\bm\a}(\bma^T\bm x)\right\}^2\,dG(\bm x)=O_p\left((\log n) ^2n^{-2/3}\right).	\end{align*}	
\end{lemma}

\section{The limit theory for the SSE}
\label{section:analysis_SSE}
\setcounter{equation}{0}
In this section we derive the limit distribution of the SSE introduced above. %a profile least squares estimator which is $\sqrt{n}$ convergent and asymptotically normal. It is asymptotically equivalent to the estimator SSE (Simple Score Estimator) in \cite{FGH:19} and  we give it the same name. 
In our derivation, the function $\psi_{\bma}$ of Definition \ref{def_psi_alpha} plays a crucial role.  Below, we will use the following assumptions, additionally to (A1) to (A6).

\begin{enumerate}
\item[(A7)] There exists a $\d>0$ such that for all $\bma\in(\B(\bma_0,\d)\cap\cS_{d-1})\setminus\{\bma_0\}$ the random variable
\begin{align*}
\text{cov}\left((\bma_0-\bma)^T\bmX,\psi_0(\bma_0^T\bmX)\bigm|\bma^T\bmX\right)
\end{align*}
is not equal to $0$ almost surely.
\item[(A8)] The matrix
\begin{align*}
\E\left[\psi_0'(\bma_0^T\bmX)\,\text{cov}(\bmX|\bma_0^T\bmX)\right]
\end{align*}
has rank $d-1$.
\end{enumerate}

We start by comparing (\ref{SSE_criterion}) with the function
\begin{align}
\label{SSE_criterion_theoretical}
\bma\mapsto\left\|\E\left\{Y-\psi_{\bma}(\bma^T\bmX)\right\}\bmX\right\|^2.
\end{align}
As in Section \ref{section:intro}, the function $\hat\psi_{n,\bma}$ is just the (isotonic) least squares estimate for fixed $\bma$.

\vspace{0.3cm}
\begin{example}[Continuation of Example \ref{example_uniform_covariates}]
\label{example_uniform_covariates2}
{\rm We consider the loss function given by
\begin{align}
\label{loss_functions2a}
L^{\text{SSE}}:\a_1\mapsto\left\|\E\left\{Y-\psi_{\bma}(\bma^T\bmX)\right\}\bmX\right\|^2,
\end{align}
and compare this with the loss function
\begin{align}
\label{loss_functions2b}
\widehat L^{\text{SSE}}_n:\a_1\mapsto \left\|n^{-1}\sum_{i=1}^n\left\{Y_i-\hat\psi_{n,\bma}(\bma^T\bmX_i)\right\}\bmX_i\right\|^2,
\end{align}
for the same data as in Example \ref{example_uniform_covariates} in Section \ref{section:psi_alpha}.
If we plot the loss functions $L^{\text{SSE}}$ and $\widehat L^{\text{SSE}}_n$ for the model of Example  \ref{example_uniform_covariates},  where $\bma=(\a_1,\a_2)^T$, for $\a_1\in[0,1]$ and $\a_2$ the positive root $\sqrt{1-\a_1^2}$, we get Figure \ref{fig:empirical_vs_analytic_SSE}. The function $L^{\text{LSE}}$ has a minimum equal to $0$ at $\a_1=1/\sqrt{2}$ while $\widehat L^{\text{SSE}}_n$ attains its minimum at a value that is very close to $1/\sqrt 2$.  

\begin{figure}[!h]
	\centering
	\begin{subfigure}{0.45\linewidth}
	\includegraphics[width=0.95\textwidth]{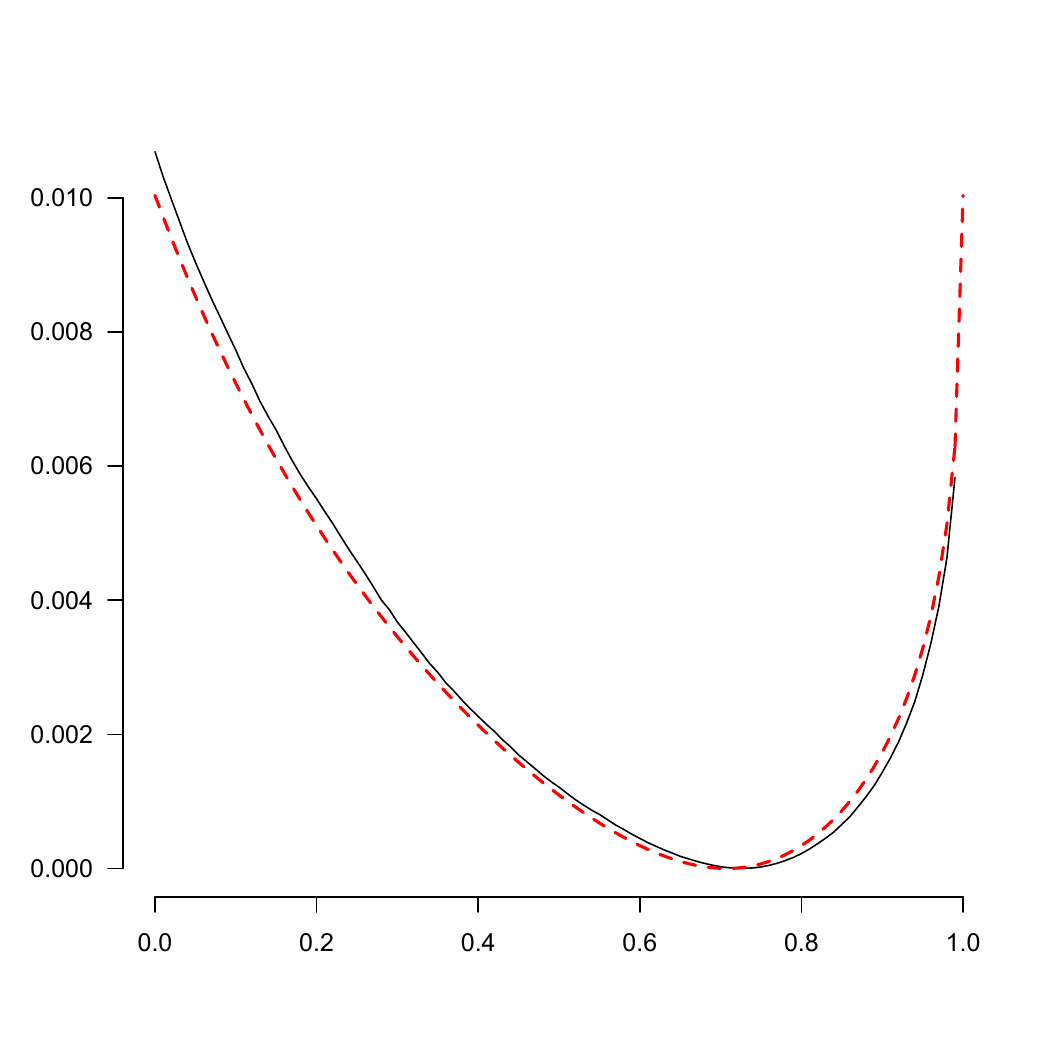}
	\caption{$n = 10^4$}
	\end{subfigure}
	\begin{subfigure}{0.45\linewidth}
	\includegraphics[width=0.95\textwidth]{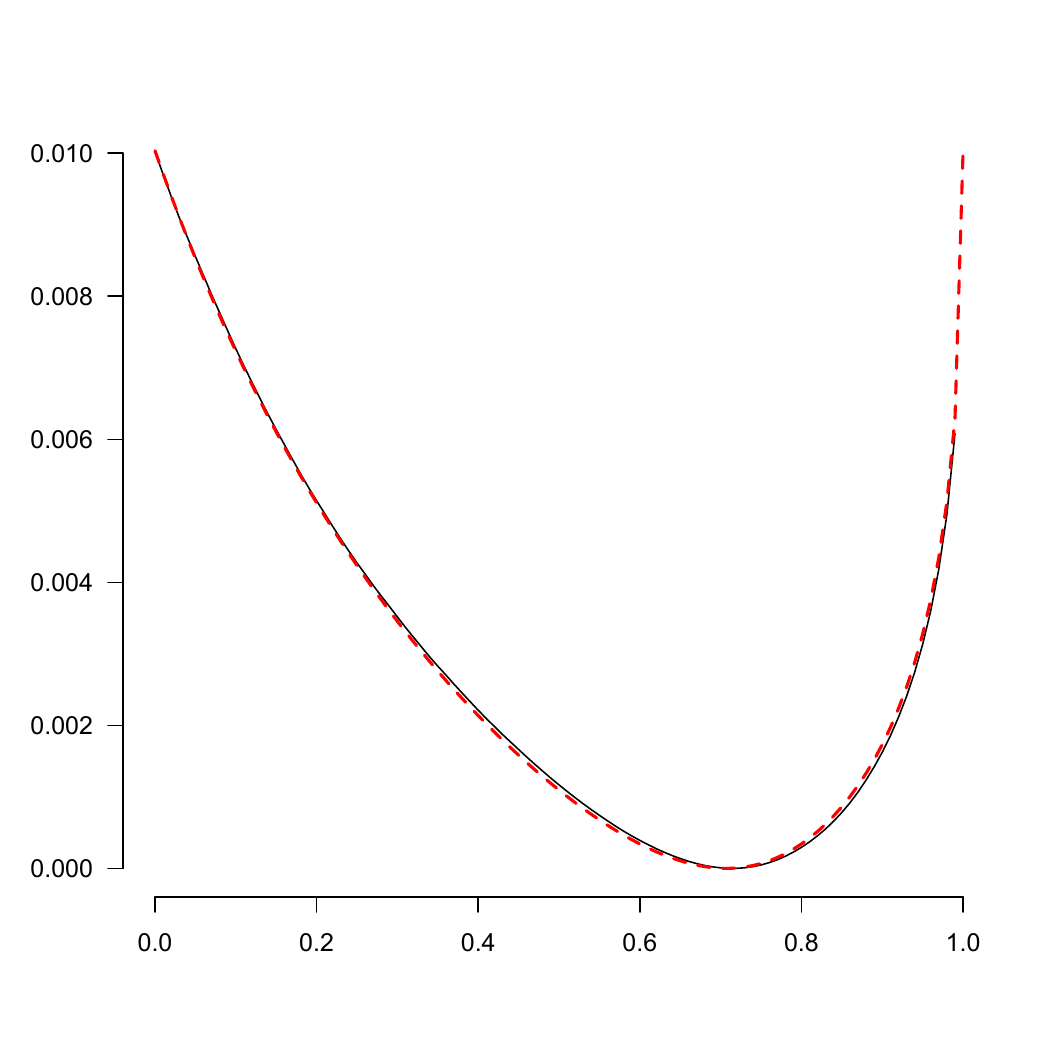}
	\caption{$n =10^5$}
\end{subfigure}\\
\caption{The loss functions $L^{\text{SSE}}$ (red, dashed) and $\widehat L^{\text{SSE}}_n$ (solid), where $n=10^4$ and $n=10^5$.}
\label{fig:empirical_vs_analytic_SSE}
	\end{figure}

In general, the curve $\widehat L^{\text{SSE}}_n$ will be smoother than the curve $\widehat L^{\text{LSE}}_n$. The rather striking difference in smoothness of the loss functions $\widehat L^{\text{LSE}}_n$ and $\widehat L^{\text{SSE}}_n$ can be seen in Figure \ref{fig:loss_on_0.65_0.8}, where we zoom in on the interval $[0.65,0.80]$ for $n=10,000$ and the examples of Figure \ref{fig:empirical_vs_analytic} and Figure \ref{fig:empirical_vs_analytic_SSE}.  The question is whether this difference in smoothness explains why the SSE is $\sqrt n$-consistent while this might not be the case for the profile LSE.}
\end{example}

\begin{figure}[!h]
	\centering
	\begin{subfigure}{0.45\linewidth}
	\includegraphics[width=0.95\textwidth]{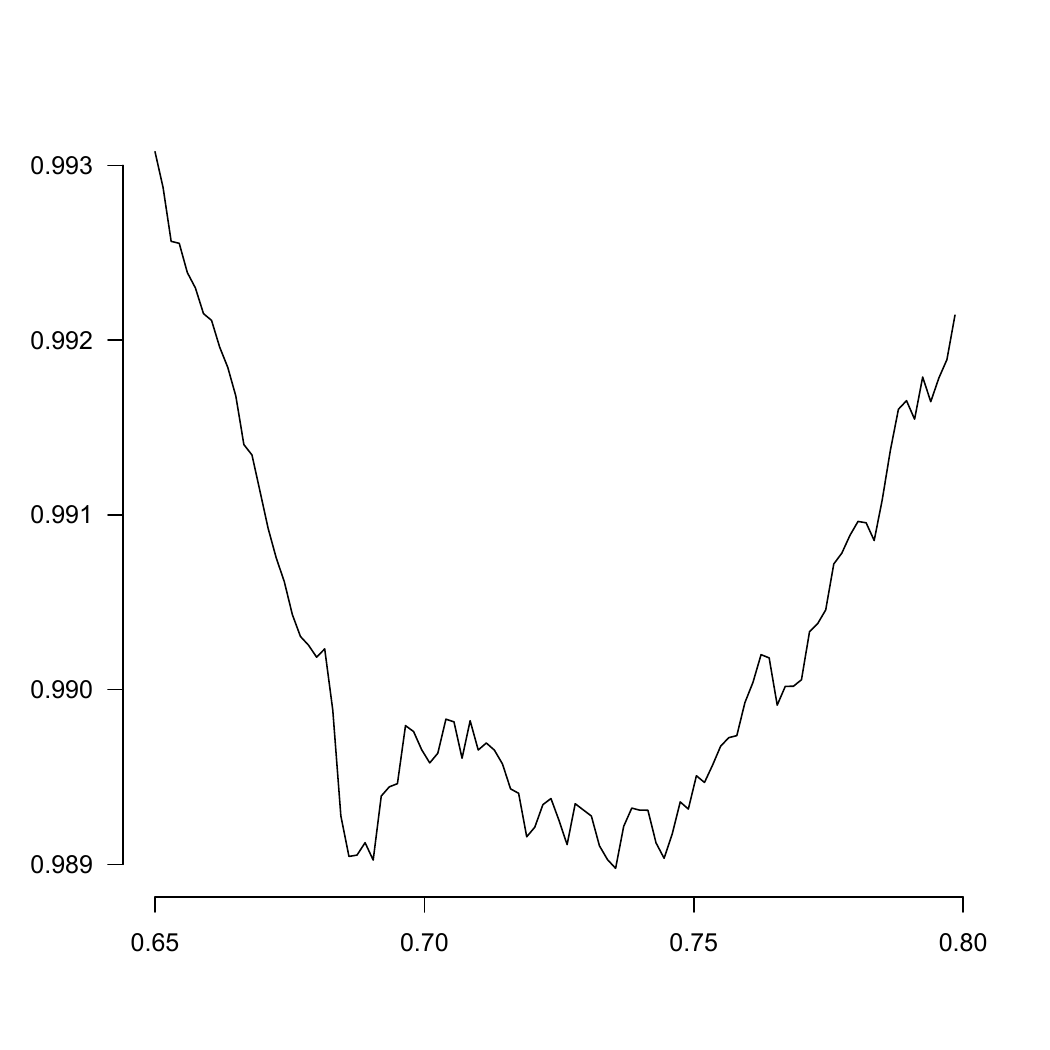}
	\caption{$\widehat L^{\text{LSE}}_n$}
	\end{subfigure}
	\begin{subfigure}{0.45\linewidth}
	\includegraphics[width=0.95\textwidth]{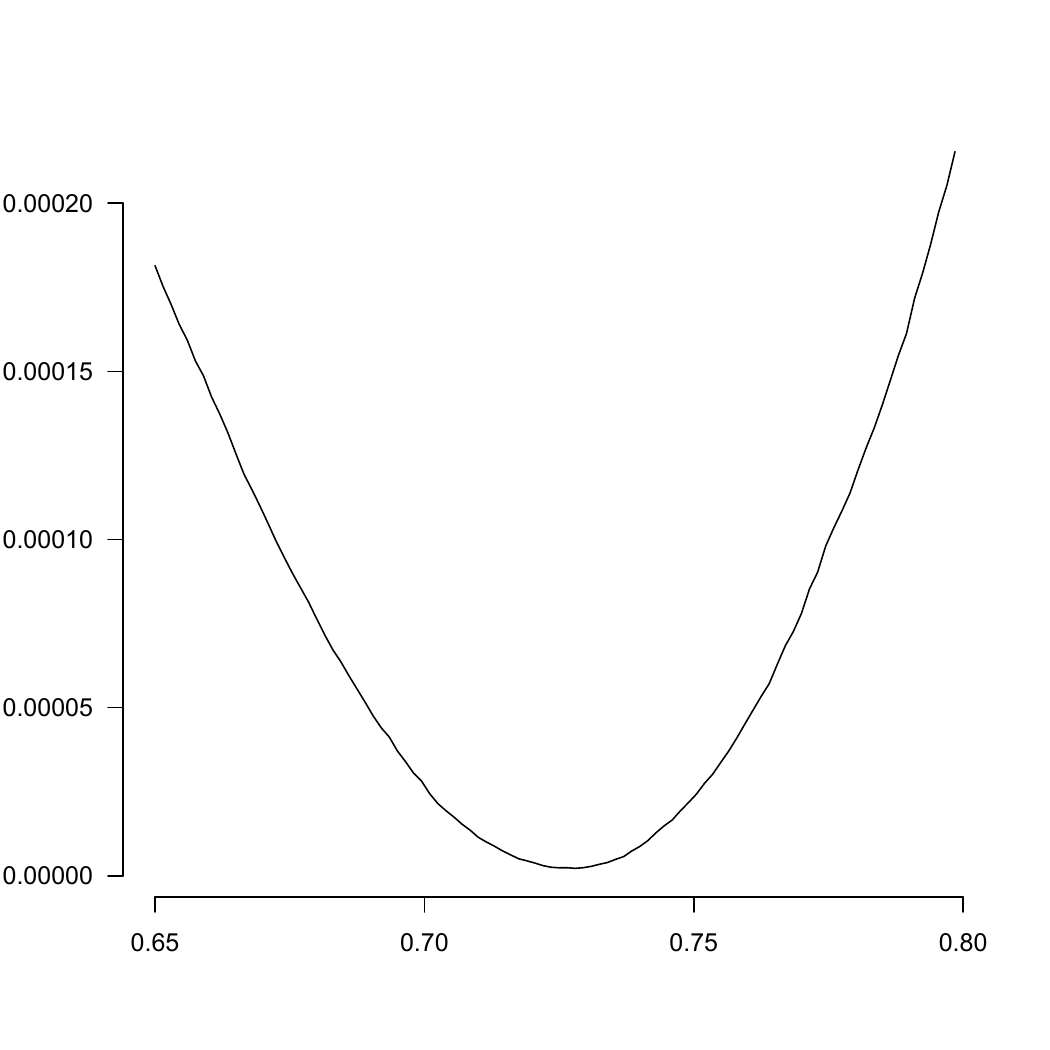}
	\caption{$\widehat L^{\text{SSE}}_n$}
\end{subfigure}\\
\caption{The loss functions $\widehat L^{\text{LSE}}_n$ and $\widehat L^{\text{SSE}}_n$ on $[0.65,0.80]$, for $n=10^4$.}
\label{fig:loss_on_0.65_0.8}
	\end{figure}

\vspace{0.3cm}
In the computation of the SSE, we have to take a starting point. For this we use the LSE, which is proved to be consistent in \cite{BDJ:19}.
The proof of the consistency of the SSE is a variation on the proof for corresponding crossing of zero estimator in \cite{FGH:19} in (D.2) of the supplementary material. We use the following lemma, which is a corollary to Proposition 2 in the supplementary material of \cite{FGH:19}.

\begin{lemma}
\label{lemma:phi_n-phi}
Let $\f_n$ and $\f$ be defined by
\begin{align*}
\f_n(\bma)=\int\bmx\left\{y-\hat\psi_{n,\bma}(\bma^T\bmx)\right\}\,d\P_n(\bmx,y),
\end{align*}
and
\begin{align*}
\f(\bma)=\int\bmx\left\{y-\psi_{\bma}(\bma^T\bmx)\right\}\,dP(\bmx,y).
\end{align*}
Then, uniformly for $\bma$ in a neighborhood $\B(\bma_0,\d)\cap\cS_{d-1}$ of $\bma_0$:
\begin{align*}
\f_n(\bma)=\f(\bma)+o_p(1).
\end{align*}
\end{lemma}

\begin{remark}
{\rm
The proof in \cite{FGH:19} used reparameterization, but this is actually not needed in the proof.
}
\end{remark}

\begin{theorem}[Consistency of the SSE]
Let $\hat\bma_n\in\cS_{d-1}$ be the SSE of $\bma_0$ and let conditions (A1) to (A8) be satisfied. Then
\begin{align*}
\hat\bma_n\stackrel{p}\longrightarrow\bma_0.
\end{align*}
\end{theorem}

\begin{lemma}
\label{lemma:transition_psi_n-psi_alpha}
Let $\hat\bma_n\in\cS_{d-1}$ be a minimizer of
\begin{align}
\label{crossing_of_zero}
\left\|n^{-1}\sum_{i=1}^n\left\{Y_i-\hat\psi_{n,\bma}(\bma^T\bmX_i)\right\}\bmX_i\right\|^2,
\end{align}
for $\bma\in\cS_{d-1}$, where $\|\cdot\|$ denotes the Euclidean norm. Then, under conditions (A1) to (A8) we have:
\begin{align}
\label{SSE_root_psi_alpha}
n^{-1}\sum_{i=1}^n\left\{Y_i-\hat\psi_{n,\hat\bma_n}(\hat\bma_n^T\bmX_i)\right\}\bmX_i
=n^{-1}\sum_{i=1}^n\left\{Y_i-\psi_{\hat\bma_n}(\hat\bma_n^T\bmX_i)\right\}\left\{\bmX_i-\E\left(\bmX|\hat\bma_n^T\bmX_i\right)\right\}+o_p\left(n^{-1/2}\right).
\end{align}
\end{lemma}

We now have the following limit result.

\begin{theorem}[Asymptotic normality of the SSE]
\label{th:asympt_normality_SSE}
Let $\hat\bma_n$ be the minimizer of
\begin{align}
\label{crossing_of_zero}
\left\|n^{-1}\sum_{i=1}^n\left\{Y_i-\hat\psi_{n,\bma}(\bma^T\bmX_i)\right\}\bmX_i\right\|^2,
\end{align}
for $\bma\in\cS_{d-1}$, where $\|\cdot\|$ denotes the Euclidean norm.
		Let the matrices $\bm A$ and $\bm\Sigma$ be defined by:
		\begin{align}
		\label{def_A}
		\bm A=\E\Bigl[\psi_0'(\bma_0^T\bm X)\,\text{\rm Cov}(\bm X|\bm\a_0^T\bm X)\Bigr],
		\end{align}
		and
		\begin{align}
		\label{def_Sigma}
		\bm \Sigma=\E\left[\left\{Y -\psi_0(\bm\a_0^T\bm X)\right\}^2\,\left\{\bm X -\E(\bm X|\bm\a_0^T\bm X) \right\}\left\{\bm X -\E(\bm X|\bm\a_0^T\bm X) \right\}^T\right].
		\end{align}
Then, under conditions $(A1)$ to $(A8)$ we have:	 
		\begin{align*}
		\sqrt n (\hat \bma_n - \bma_0) \to_d N\left(\bm 0,  \bm A^- \bm \Sigma \bm A^-\right),
		\end{align*}
		where $\bm A^{-}$ is the Moore-Penrose inverse of $\bm A$.
\end{theorem}

\vspace{0.5cm}
\begin{example}[Continuation of Example \ref{example_uniform_covariates2}] 
{\rm We compute the asymptotic covariance matrix for Example \ref{example_uniform_covariates2}. In this case we get for matrix $\bm A$ in part (ii) of Theorem \ref{th:asympt_normality_SSE}:
\begin{align*}
\bmA&=\E\Bigl[\psi_0'(\bma_0^T\bm X)\,\text{\rm Cov}(\bm X|\bm\a_0^T\bm X)\Bigr]\\
&=\frac34\E\left[\left(\frac{X_1+X_2}{\sqrt{2}}\right)^2\left(\bmX-\E(\bmX|\bma_0^T\bmX)\right)\left(\bmX-\E(\bmX|\bma_0^T\bmX)\right)^T\right]\\
&=\left(\begin{array}{rrr}
1/15 &-1/15\\
-1/15	&1/15
\end{array}
\right).
\end{align*}
The Moore-Penrose inverse of $\bmA$ is given by:
\begin{align*}
\bmA^-=\left(\begin{array}{rrr}
15/4 &-15/4\\
-15/4	&15/4
\end{array}
\right).
\end{align*}
Furthermore, we get:
\begin{align*}
\bm \Sigma&=\E\left[\left\{Y -\psi_0(\bm\a_0^T\bm X)\right\}^2\,\left\{\bm X -\E(\bm X|\bm\a_0^T\bm X) \right\}\left\{\bm X -\E(\bm X|\bm\a_0^T\bm X) \right\}^T\right]\\
&=\E\left\{\bm X -\E(\bm X|\bm\a_0^T\bm X) \right\}\left\{\bm X -\E(\bm X|\bm\a_0^T\bm X) \right\}^T\\
&=\left(\begin{array}{rrr}
1/24 &-1/24\\
-1/24	&1/24
\end{array}
\right).
\end{align*}
So the asymptotic covariance matrix is given by:
\begin{align*}
\bmA^-\bm\Sigma\bmA^-=\left(\begin{array}{rrr}
75/32 &-75/32\\
-75/32	&75/32
\end{array}
\right)
\approx
\left(\begin{array}{rrr}
2.34375 &-2.34375\\
-2.34375	&2.34375
\end{array}
\right).
\end{align*}
}
\end{example}

\vspace{0.3cm}
\begin{remark}
{\rm Theorem \ref{th:asympt_normality_SSE} corresponds to Theorem 3 in \cite{FGH:19}, but note that the estimator has a different definition. Reparameterization is also avoided.
}
\end{remark}

\section{The limit theory for the ESE and cubic spline estimator}
\label{sec:ESE_PLSE}
\setcounter{equation}{0}

The proofs of the consistency and asymptotic normality of the ESE and spline estimator are highly similar to the proofs of these facts for the SSE in the preceding section. The only extra ingredient is occurrence of the estimate of the derivative of the link function. We only discuss the asymptotic normality.

In addition to the assumptions (A1) to (A7), we now assume:
\begin{enumerate}
\item[(A8')] $\psi_{\bma}$ is twice differentiable on $(\inf_{x\in{\cal X}}(\bma^T\bmx),\sup_{x\in{\cal X}}(\bma^T\bmx))$.
\item[(A9)] The matrix
\begin{align*}
\E\left[\psi_0'(\bma_0^T\bmX)^2\,\text{cov}(\bmX|\bma_0^T\bmX)\right]
\end{align*}
has rank $d-1$.
\end{enumerate}

An essential step is again to show that
\begin{align*}
&\int\bmx\left\{y-\hat\psi_{n,\hat\bma_n}(\hat\bma_n^T\bmx)\right\}\hat\psi'_{n\hat\bma_n}(\hat\bma_n^T\bmx)\,d\P_n(\bmx,y)\\
&=\int\left\{\bmx-\E(X|\hat\bma_n^T\bmX)\right\}\left\{y-\hat\psi_{n,\hat\bma_n}(\hat\bma_n^T\bmx)\right\}\hat\psi'_{n\hat\bma_n}(\hat\bma_n^T\bmx)\,d\P_n(\bmx,y)
+o_p(n^{-1/2})+o_p(\hat\bma_n-\bma_0),
\end{align*}
For the ESE this is done by defining the piecewise constant function $\bar{\rho}_{n, \bma}$ for $u$ in the interval between successive jumps $\tau_i$  and $\tau_{i+1})$  of $\hat\psi_{n\bma}$ by:
\begin{eqnarray*}
	\bar {\rho}_{n, \bma}(u)  =  \left \{
	\begin{array}{lll}
		\E[\bm X| \bma^T\bm X= \t_i]\psi_{\bma}'(\t_i) \ \  \ \ \ \ \ \ \textrm{ if $\psi_{\bma}(u)  > \hat\psi_{n\bma}(\tau_i)$  \ for all $u \in (\tau_i, \tau_{i+1})$}, \\
		\E[\bm X| \bma^T\bm X= s]\psi_{\bma}'(s) \ \ \ \  \ \  \  \ \ \ \textrm{ if $\psi_{\bma}(s)  = \hat\psi_{n\bma}(s)$ \ for some $s \in (\tau_i, \tau_{i+1})$}, \\
		\E[\bm X| \bma^T\bm X= \t_{i+1}]\psi_{\bma}'(\t_{i+1})\ \ \ \textrm{if $\psi_{\bma}(u) < \hat\psi_{n\bma}(\tau_i)$  \ for all $u \in (\tau_i, \tau_{i+1})$}. 
	\end{array}
	\right.
\end{eqnarray*}
where $\bar{\rho}_{n,\bma}$ replaces $\bar E_{n, \bma}$ in (\ref{E_n-def}), see Appendix E in the supplement of \cite{FGH:19}. The remaining part of the proof runs along the same lines as the proof for the SSE. For additional details, see Appendix E in the supplement of \cite{FGH:19}.

The corresponding step in the proof for the spline estimator is given by the following lemma.

\begin{lemma}
\label{lemma:spline_lemma}
Let the conditions of Theorem 5 in  \cite{arun_rohit:20} be satisfied. In particular, let the penalty parameter $\m_n$ satisfy $\m_n=o_p(n^{-1/2})$. Then we have for all $\bma$ in a neighborhood of $\bma_0$ and for the corresponding natural cubic spline  $\hat\psi_{n\bma}$:
\begin{align*}
\int\E(\bmX|\bma^T\bmX)\left\{y-\hat\psi_{n\bma}\left(\bma^T\bmx\right)\right\}\hat\psi_{n\bma}'\left(\bma^T\bmx\right)\,d\P_n(\bmx,y)=O_p(\mu_n)=o_p\left(n^{-1/2}\right).
\end{align*}
\end{lemma}

\begin{remark}
{\rm The result shows that we have as our  basic equation in $\bma$:
\begin{align*}
&\frac1n\sum_{i=1}^n \bigl\{\hat\psi_{n\bma}(\bma^T\bmX_i)-Y_i\bigr\}\hat\psi'_{n\bma}(\bma^T\bmX_i)\bmX_i\\
&=\frac1n\sum_{i=1}^n \bigl\{\hat\psi_{n\bma}(\bma^T\bmX_i)-Y_i\bigr\}\hat\psi'_{n\bma}(\bma^T\bmX_i)\left\{\bmX_i-\E(\bmX_i|\bma^T\bmX_i)\right\}+o_p\left(n^{-1/2}\right)\\
&=o_p\left(n^{-1/2}\right).
\end{align*}
}
\end{remark}

The remaining part of the proof of the asymptotic normality can either run along the same lines as the proof for the corresponding fact for the SSE, using the function $u\mapsto\psi_{\bma}(u)=\E\{\psi_0(\bma^T\bmx)|\bma^T\bmX=u\}$, or directly use the convergence of $\hat\psi_{n\hat\bma_n}$ to $\psi_0$ and of $\hat\psi'_{n\hat\bma_n}$ to $\psi_0'$ (see Theorem 3 in \cite{arun_rohit:20}). For the SSE and ESE we were forced to introduce the intermediate function $\psi_{\bma}$ to get to the derivatives, because for these estimators the derivative of $\hat\psi_{n\hat\bma_n}$ did not exist.

We get the following result.

\begin{theorem}
	\label{theorem:asymptotics-efficient}
	Let either $\hat\bma_n$ be the ESE of $\bma_0$ and  let Assumptions (A1) to (A7) and (A8') and (A9) of the present section be satisfied or let $\hat\bma_n$ be the spline estimator of $\bma_0$ and let Assumptions (A0) to (A6) and (B1) to (B3) of  \cite{arun_rohit:20}) be satisfied.
	Moreover, let the bandwidth $h \asymp n^{-1/7}$ in the estimate of the derivative of $\psi_{\bma}$ for the ESE.
			Define the matrices,
		\begin{align}
		\label{def:I1}
		\tilde {\bm A}:=\E\Bigl[\psi_0'(\bm\a_0^T\bm X)^2\,\text{\rm Cov}(\bm X|\bm\a_0^T\bm X)\Bigr],
		\end{align}
		and
		\begin{align}
		\label{def:I2}
		\tilde {\bm \Sigma}:=\E\left[\left\{Y -\psi_0(\bm\a_0^T\bm X)\right\}^2\psi_0'(\bm\a_0^T\bm X)^2\left\{\bm X -\E(\bm X|\bm\a_0^T\bm X) \right\}\left\{\bm X -\E(\bm X|\bm\a_0^T\bm X) \right\}^T\right].
		\end{align}
		Then 
		\begin{align*}
		\sqrt n (\tilde \bma_n - \bma_0) \to_d N_{d}\left(\bm 0,  \tilde {\bm A}^- \tilde {\bm \Sigma} \tilde {\bm A}^-\right),
		\end{align*}
		where $\tilde {\bm A}^{-}$ is the Moore-Penrose inverse of $\tilde {\bm A}$.	
\end{theorem}

This corresponds to Theorem 6 in \cite{FGH:19} and Theorem 5 in \cite{arun_rohit:20}), but note that the formulation of Theorem 5 in \cite{arun_rohit:20} still contains  the Jacobian connected with the lower dimensional parameterization. Consequently, the ESE and the cubic spline estimator admit the same weak limit under the conditions stated above.

\section{Simulation and comparisons with other estimators}
\label{section:simulations}
\setcounter{equation}{0}

In this section we compare the LSE with the Simple Score Estimator (SSE), the Efficient Score Estimator (ESE), the Effective Dimension Reduction (EDR) estimate, the spline estimate, the MAVE estimate and the EFM estimate. We take part of the simulation settings in \cite{BDJ:19}, which means that we take the dimension $d$ equal to $2$. Since the parameter belongs to the boundary of a circle in this case, we only have to determine a $1$-dimensional parameter. Using this fact, we use the parameterization $\bma=(\a_1,\a_2)=(cos(\b),\sin(\b))$ and determine the angle $\b$ by a golden section search for the SSE, ESE and spline estimate.  For the EDR we used the {\tt R} package {\tt edr}; the method is discussed in \cite{hristache01}. The spline method is described in \cite{arun_rohit:20}, and there exists an {\tt R} package {\tt simest} for it, but we used our own implementation.
For the MAVE method we used the {\tt R} package {\tt MAVE}, for theory see \cite{xia:06}. For the EFM estimate (see \cite{cui2011}) we used an {\tt R} script, due to Xia Cui and kindly provided to us by her and Rohit Patra. All runs of our simulations can be reproduced by running the {\tt R} scripts in \cite{github:18}.

\begin{figure}[!h]
	\centering
	\begin{subfigure}{0.45\linewidth}
	\includegraphics[width=0.95\textwidth]{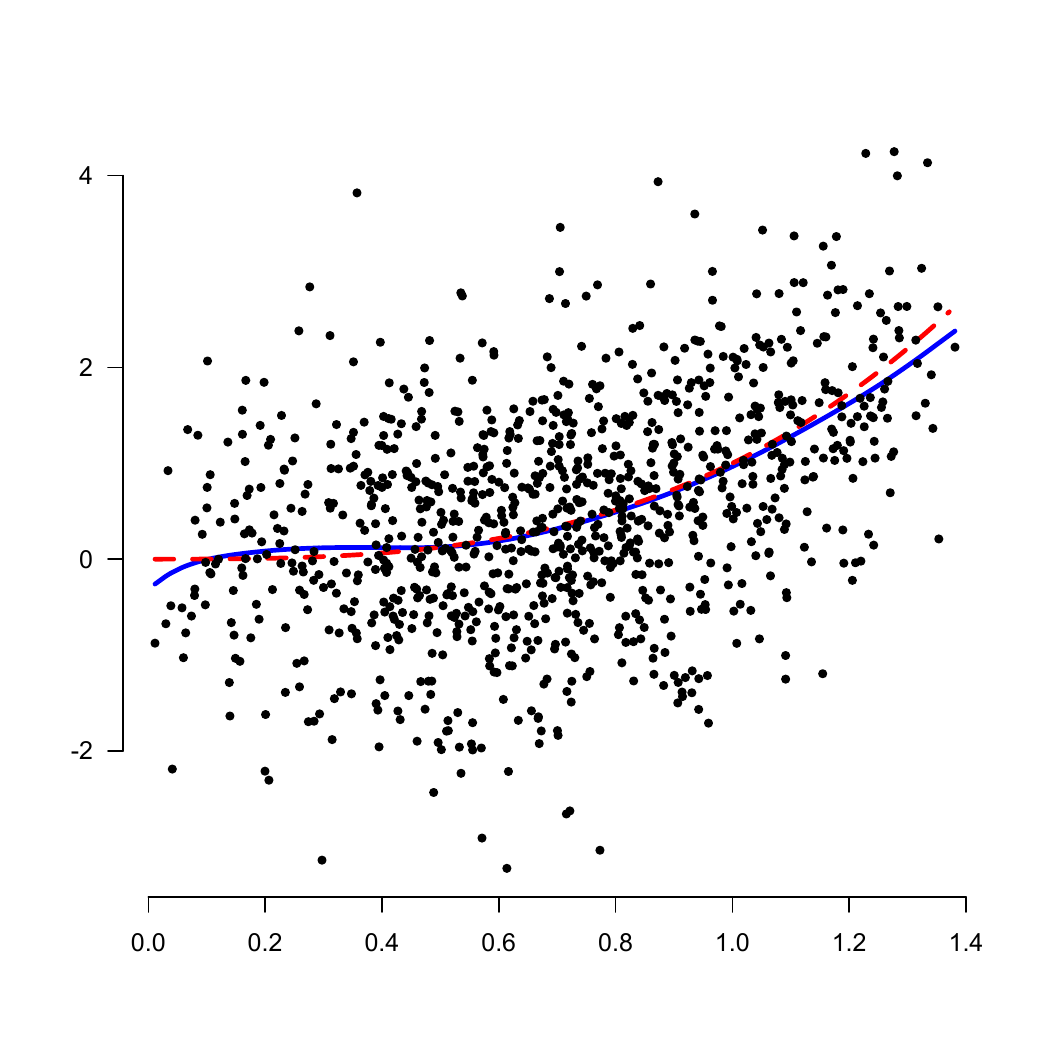}
	\caption{$\hat\bma_n=(0.71055,0.70364)$}
	\end{subfigure}
	\begin{subfigure}{0.45\linewidth}
	\includegraphics[width=0.95\textwidth]{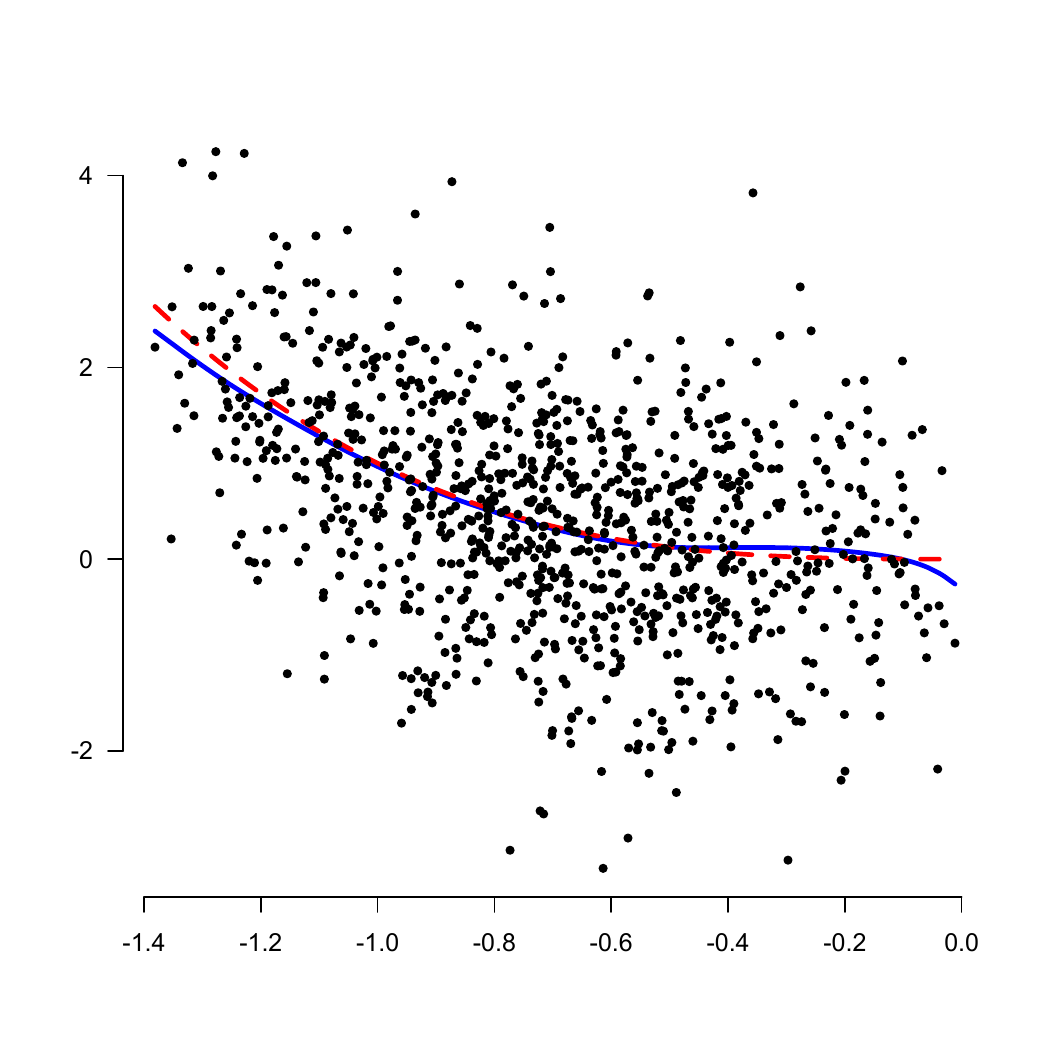}
	\caption{$\hat\bma_n=-(0.71055,0.70364)$}
\end{subfigure}\\
\caption{Two MAVE estimates of $\bma_0=2^{-1/2}(1,1)^T$ for model 1 with sample size $n=1000$: (a) from starting the iterations at $\bma_0$, (b) from starting the iterations at $-\bma_0$; the blue solid curve is the estimate of the link function, based on $\hat\bma_n$; the blue dashed function is $t\mapsto t^3$ in (a) and $t\mapsto -t^3$ in (b). Note that in (b) also the sign of the first coordinates of the points $(\hat\bma_n^T\bmX_i,Y_i)$ in the scatterplot is reversed. Under the restriction that the link function is nondecreasing (b) cannot be a solution.}
\label{fig:Mave_example}
	\end{figure}

In simulation model 1 we take $\bma_0=(1/\sqrt{2},1/\sqrt{2})^T$ and $\bmX=(X_1,X_2)^T$, where $X_1$ and $X_2$ are independent Uniform$(0,1)$ variables. The model is now:
\begin{align*}
Y=\psi_0(\bma_0^T\bmX)+\e,
\end{align*}
where $\psi_0(u)=u^3$ and $\e$ is a standard normal random variable, independent of $\bmX$.

In simulation model 2 we also take $\bma_0=(1/\sqrt{2},1/\sqrt{2})^T$ and $\bmX=(X_1,X_2)^T$, where $X_1$ and $X_2$ are independent Uniform$(0,1)$ variables.  This time, however, the model is:
\begin{align*}
Y=\text{Bin}\left(10,\exp(\bma_0^T\bmX)/\left\{1+\exp(\bma_0^T\bmX)\right\}\right),
\end{align*}
see also Table 2 in \cite{BDJ:19}.
This means:
\begin{align*}
Y=\psi_0(\bma_0^T\bmX)+\e,
\end{align*}
where
\begin{align*}
\psi_0(\bma_0^T\bmX)=10\,\exp(\bma_0^T\bmX)/\{1+\exp(\bma_0^T\bmX)\},\qquad \e=N_n-\psi_0(\bma_0^T\bmX),
\end{align*}
and
\begin{align*}
N_n=\text{Bin}\left(10,\frac{\exp(\bma_0^T\bmX}{1+\exp(\bma_0^T\bmX)}\right).
\end{align*}
Note that indeed $\E\{\e|\bmX)=0$,
but that we do not have independence of $\e$ and $\bmX$, as in the previous example.

It was noticed in \cite{xia:06}, p.\ 1113, that, although it was shown in \cite{hristache01} that the $\sqrt{n}$ rate of convergence for the estimation of $\bma_0$ can be achieved, the asymptotic distribution of the method proposed in \cite{hristache01} was not derived, which makes it difficult to compare the limiting efficiency of the estimation method with other methods. In \cite{xia:06} the asymptotic distribution of the rMAVE estimate is derived (see Theorem 4.2 of \cite{xia:06}), which shows that this limit distribution is actually the same as that of the ESE and the spline estimate. Since Xia is one of the authors of the recent {\tt MAVE} {\tt R} package, we assume that the rMAVE method has been implemented in this package, so we will identify MAVE with rMAVE in the sequel.

The proof of the asymptotic normality result for the MAVE method uses the fact that the iteration steps, described on p.1117 of \cite{xia:06}, start in a neighborhood $\{\bma:\|\bma-\bma_0\|\le C n^{-1/2+c_0}\}$ of $\bma_0$, where $C>0$ and $c_0<1/20$, and indeed our original experiments with the {\tt R} package showed many outliers, probably due to starting values not sufficiently close to $\bma_0$. A further investigation revealed that there were many solutions in the neighborhood of the points $-\bma_0$. This phenomenon is illustrated in Figure \ref{fig:Mave_example}, generated by our own implementation of the algorithm in \cite{xia:06}. The link function is constructed from the values $a_j^{\hat\bma_n}$ in the algorithm in \cite{xia:06}, p.\ 1117, where the ordered values of $\hat\bma_n^T\bmX_j$ are the first coordinates.

Because of the difficulty we just discussed, we reversed in the results of the {\tt MAVE} {\tt R} package the sign of the solutions in the neighborhood of $-\bma_0$. By the parameterization with a positive first coordinate in \cite{cui2011} situation (b) in Figure \ref{fig:Mave_example} cannot occur for the EFM algorithm. We also tried a modification of the same type as our modification of the MAVE algorithm for the EDR algorithm, but this did not lead to a similar improvement of the results.

It follows from Theorem \ref{th:asympt_normality_SSE} that the variance of the asymptotic normal distribution for the SSE is equal to $2.727482$ and from Theorem \ref{theorem:asymptotics-efficient} that the variance of the asymptotic normal distribution for the ESE and spline estimator equals 2.737200. We already noticed in Section \ref{sec:ESE_PLSE} that the present models is not homoscedastic. In this case the asymptotic covariance matrix for the SSE of Theorem \ref{th:asympt_normality_SSE} is in fact given by
$\bmA^-=\bmA^-\bm\Sigma\bmA^-$.

\begin{figure}
	\centering
	\begin{subfigure}{0.4\linewidth}
		\includegraphics[width=1.0\textwidth]{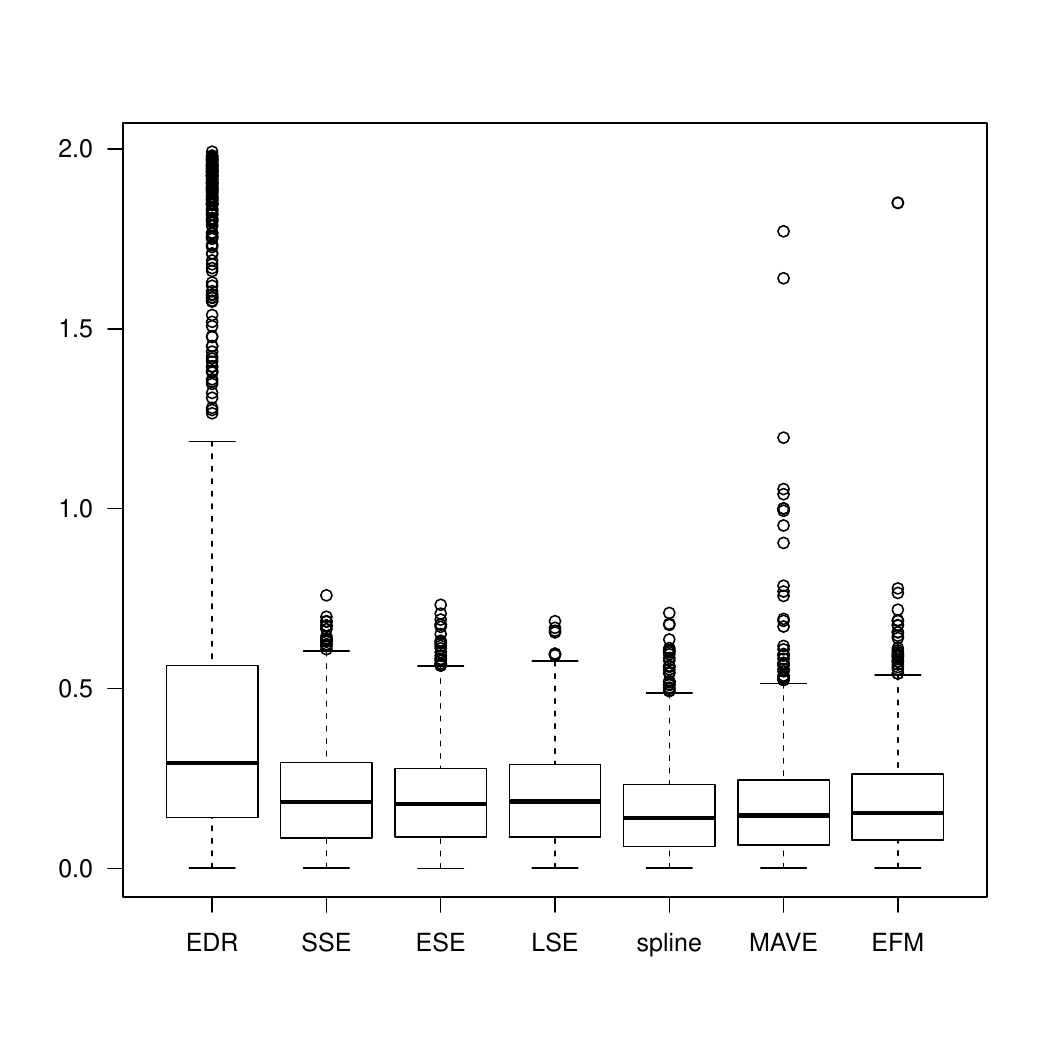}
		\caption{$n = 100$}
	\end{subfigure}
	\begin{subfigure}{0.4\linewidth}
	\includegraphics[width=1.0\textwidth]{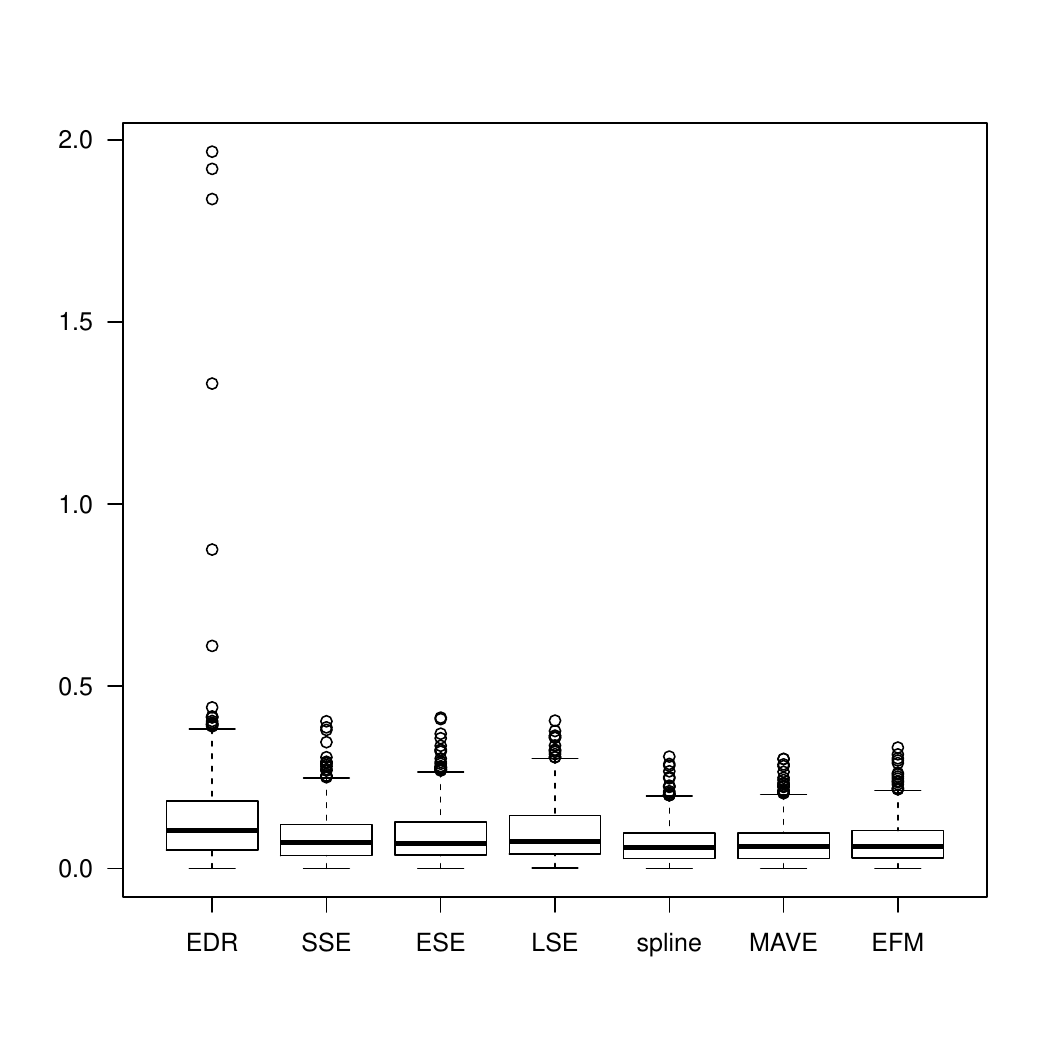}
	\caption{$n = 500$}
\end{subfigure}\\
	\begin{subfigure}{0.4\linewidth}
		\includegraphics[width=1.0\textwidth]{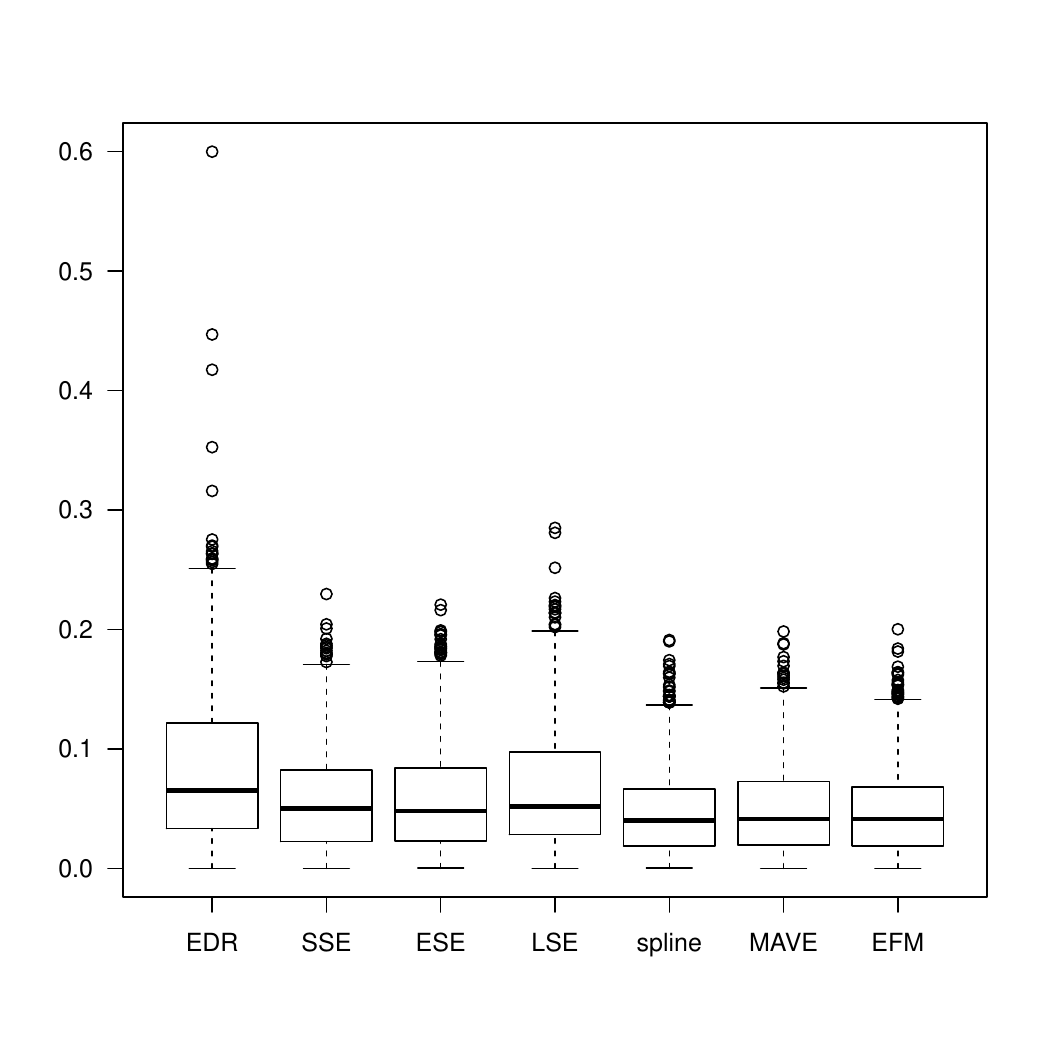}
		\caption{$n=1000$}
	\end{subfigure}
	\begin{subfigure}{0.4\linewidth}
	\includegraphics[width=1.0\textwidth]{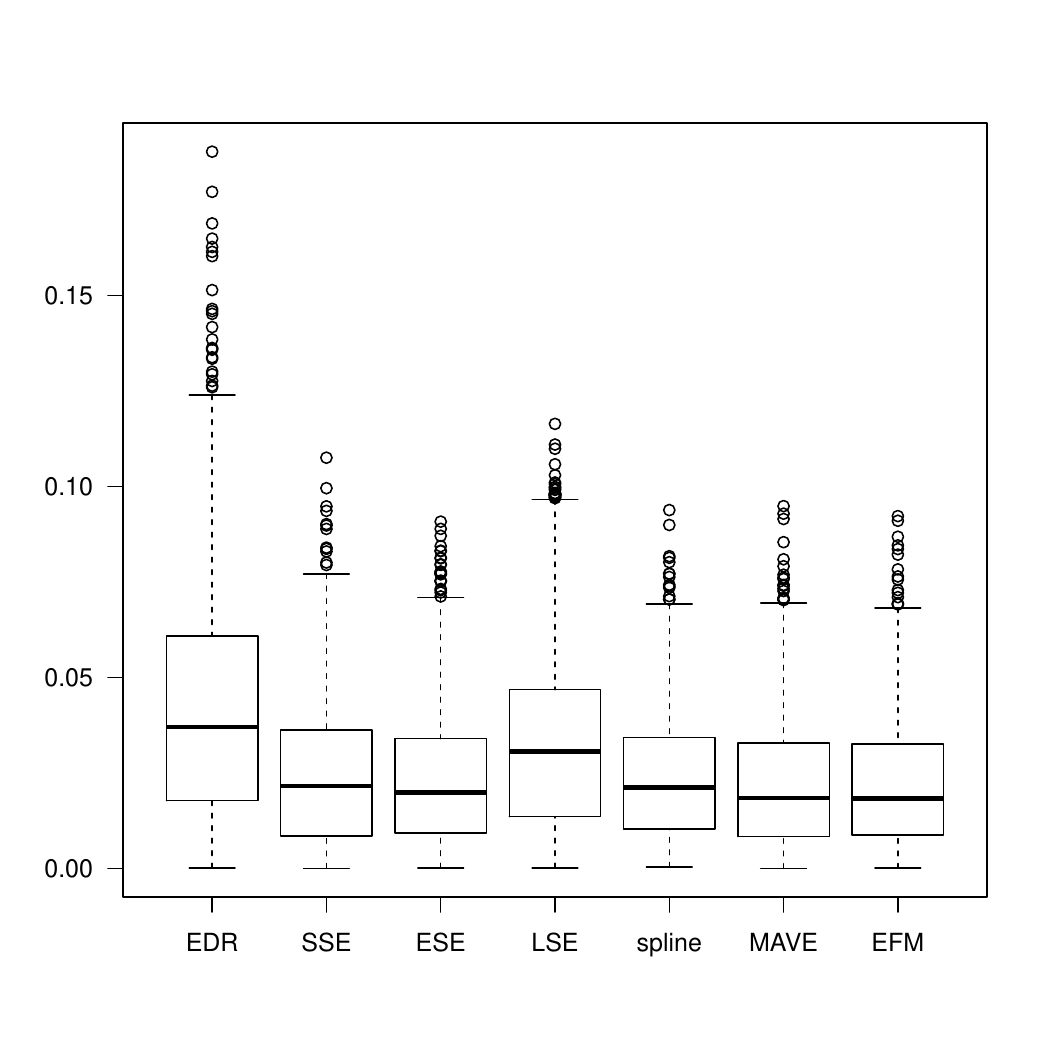}
	\caption{$n = 5000$}
	\end{subfigure}
	\caption{Boxplots of $\sqrt{n/2}\,\|\hat\bma_n-\bma_0\|_2$ for model 1.  In (b) and  (c) the values of EDR were truncated at $0.6$ to show more clearly the differences between the other estimates.}
	\end{figure}

\begin{table}
	\label{table:simulation1}
	\caption{Simulation, model 1; $\e_i$ is standard normal and independent of $\bmX_i$, consisting of two independent Uniform$(0,1)$ random variables.  The mean value $\hat\mu_i$ =  mean($\hat \a_{in}),\, i=1,2$ and $n$ times the variance-covariance $\hat\sigma_{ij}$ $=n\cdot$cov$(\hat\a_{in},\hat\a_{jn}),\,i,\,j=1,2,$\\ of the Efficient Dimension Reduction Estimate EDR, computed by the {\tt R} package {\tt edr}, the Least Squares Estimate (LSE), the Simple Score Estimate (SSE), the Efficient Score Estimate (ESE), the spline estimate, the MAVE estimate and the EFM estimate for different sample sizes $n$. The line, preceded by $\infty$, gives the asymptotic values (unknown for EDR and LSE). The values are based on $1000$ replications.}
	\vspace{0.5cm}
	\scalebox{1.0}{
		\begin{tabular}{|lr|cc|rrr |}
			\hline
			Method&$n$ & $\hat\mu_1$ & $\hat\mu_2$  & $\hat \sigma_{11}$& $\hat\sigma_{22}$& $\hat\sigma_{12}$\\						\hline
			&&&&&&\\
			EDR&100 & 0.621877 & 0.361894 &11.409222&36.869184& 9.152389  \\
			&500 & 0.701217 & 0.686094 & 7.334756 & 11.468453 & -3.881349 \\
			&1000 & 0.701669 & 0.702244  & 6.437653 & 8.090771& -3.552562 \\
			&5000& 0.706021 & 0.706798  &7.344431& 7.276717 & -7.288047\\
			\hline
			&$\infty$&0.707107 & 0.707107  & ? & ? & ?\\
		\hline
		&&&&&&\\
			LSE&100 & 0.672698 & 0.697350  &3.148912& 2.975246& -2.915427  \\
			&500 & 0.702163 & 0.701718  & 3.620960 & 3.665710 & -3.588491 \\
			&1000 & 0.704706 &0.704320  & 3.665561 & 3.664711 & -3.637541  \\
			&5000& 0.707262  & 0.705690  & 4.435842 & 4.485168 & -4.453713\\
			\hline
			&$\infty$&0.707107 & 0.707107  & ? & ? & ?\\
		\hline
		&&&&&&\\
			SSE&100 & 0.673997 & 0.6919403  &3.338637& 3.362656& -3.141408  \\
			&500 & 0.699986 & 0.706198  & 2.849647 & 2.807978 & -2.793798\\
			&1000 & 0.706477 & 0.704191  & 2.501106 & 2.510047 & -2.494237  \\
			&5000& 0.707090 & 0.706423  & 2.473765 & 2.485884 & -2.477371 \\
			\hline
			&$\infty$&0.707107 & 0.707107  & 2.343750 & 2.343750 & -2.343750\\
		\hline
		&&&&&&\\
			ESE&100 & 0.682781 & 0.687949  &3.067802& 2.991976& -2.855176  \\
			&500 & 0.702940 & 0.702462  & 3.100843& 3.116337 & -3.064151 \\
			&1000 & 0.704055 & 0.706387  & 2.676388 & 2.653164 & -2.650667  \\
			&5000& 0.707130 & 0.706444  & 2.257541 & 2.265547 & -2.259443 \\
			\hline
			&$\infty$&0.707107 & 0.707107  & 1.885522 & 1.885522 & -1.885522\\
		\hline
			&&&&&&\\
			spline&100 & 0.690741 & 0.705485  &1.801235& 1.762567& -1.711552  \\
			&500 & 0.703670 & 0.702640  & 1.795384 & 1.778454 & -1.773560 \\
			&1000 & 0.705684 &0.706007  & 1.786589 & 1.781797 & -1.777691  \\
			&5000&  0.706404 & 0.707193  & 2.180466 & 2.181544 & -2.179269 \\
			\hline
			&$\infty$&0.707107 &  0.707165  & 1.885522 & 1.885522 & -1.885522\\
		\hline
			&&&&&&\\
			MAVE&100 & 0.686503 & 0.684887  &2.423618& 3.546768& -2.245708  \\
			&500 & 0.703333 & 0.705537  & 1.897806 & 1.876220 & -2.040677 \\
			&1000 & 0.705840 &0.705660  & 1.929966 & 1.907128 & -1.911452  \\
			&5000& 0.707328 &0.706299  & 2.071168 & 2.082169 & -2.074914 \\
			\hline
			&$\infty$&0.707107 & 0.707107  & 1.885522 & 1.885522 & -1.885522\\
		\hline
			&&&&&&\\
			EFM&100 & 0.686292 & 0.684274  &2.802308& 3.280956& -2.312445  \\
			&500 & 0.703236 & 0.705133  & 2.082162 & 2.045150 & -2.044960 \\
			&1000 & 0.705629 & 0.705950  & 1.866486 & 1.860184 & -1.856340  \\
			&5000& 0.707269 & 0.707251  & 1.953800 & 1.964081 & -1.957351 \\
			\hline
			&$\infty$&0.707107 & 0.707107  & 1.885522 & 1.885522 & -1.885522\\
		\hline
		\end{tabular}}
	\end{table}

	\begin{figure}
	\centering
	\begin{subfigure}{0.4\linewidth}
		\includegraphics[width=1.0\textwidth]{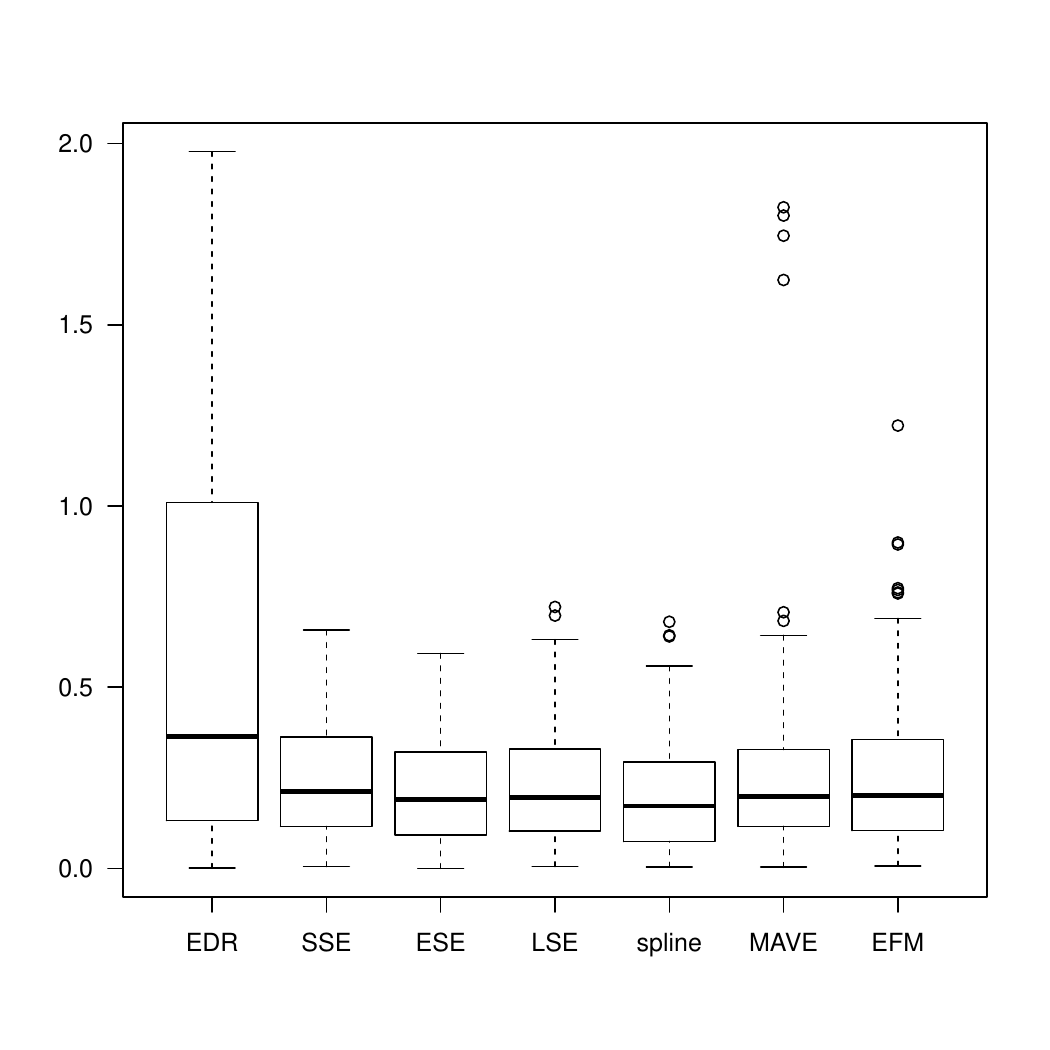}
		\caption{$n = 100$}
	\end{subfigure}
	\begin{subfigure}{0.4\linewidth}
	\includegraphics[width=1.0\textwidth]{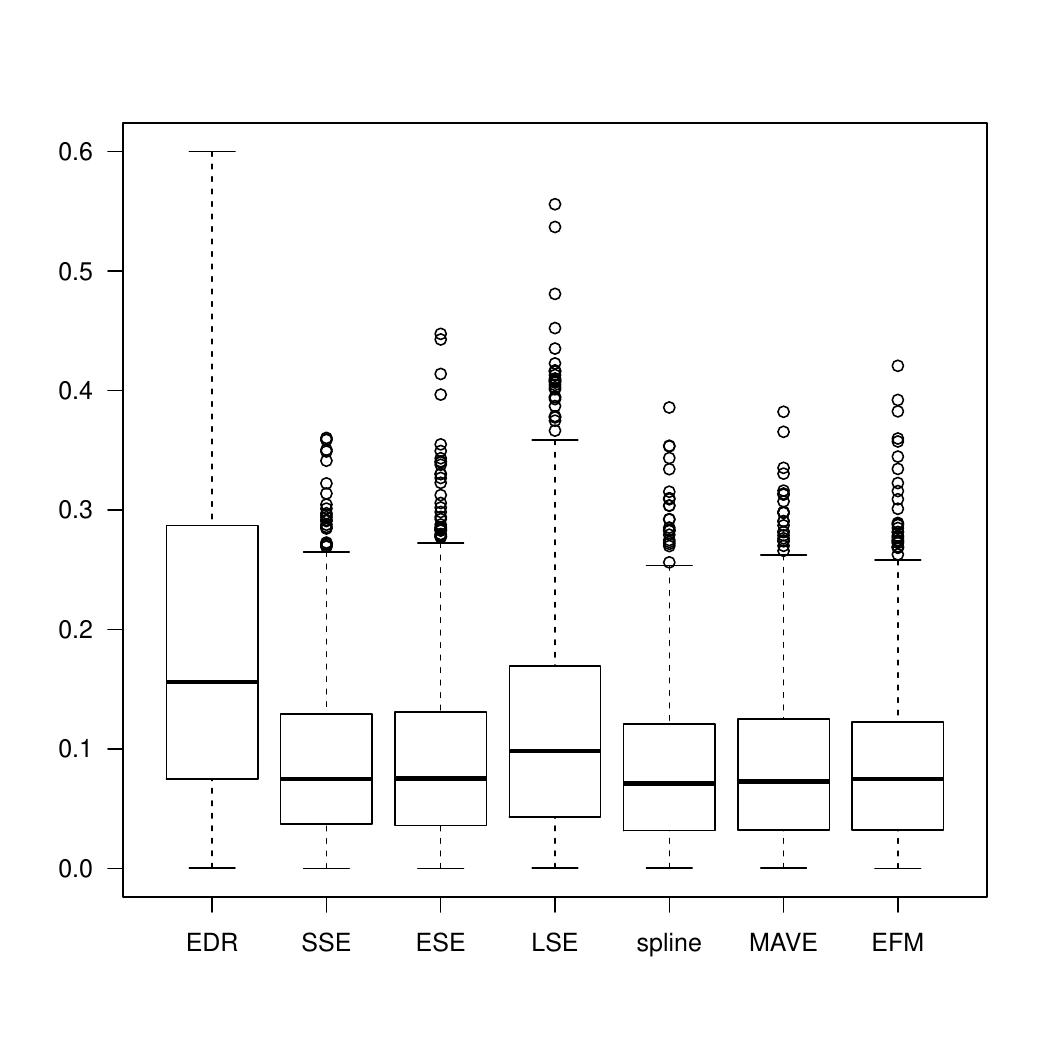}
	\caption{$n = 500$}
\end{subfigure}\\
	\begin{subfigure}{0.4\linewidth}
		\includegraphics[width=1.0\textwidth]{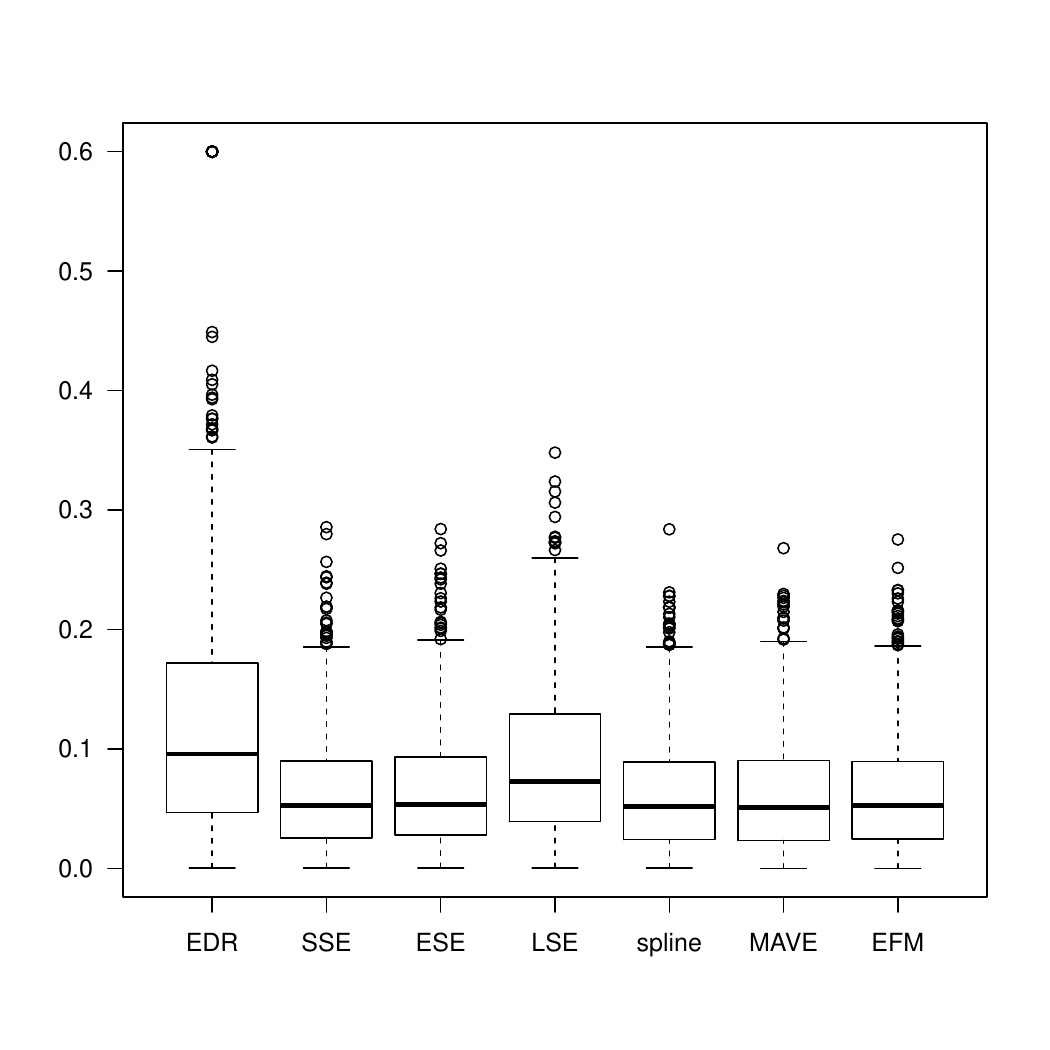}
		\caption{$n=1000$}
	\end{subfigure}
	\begin{subfigure}{0.4\linewidth}
	\includegraphics[width=1.0\textwidth]{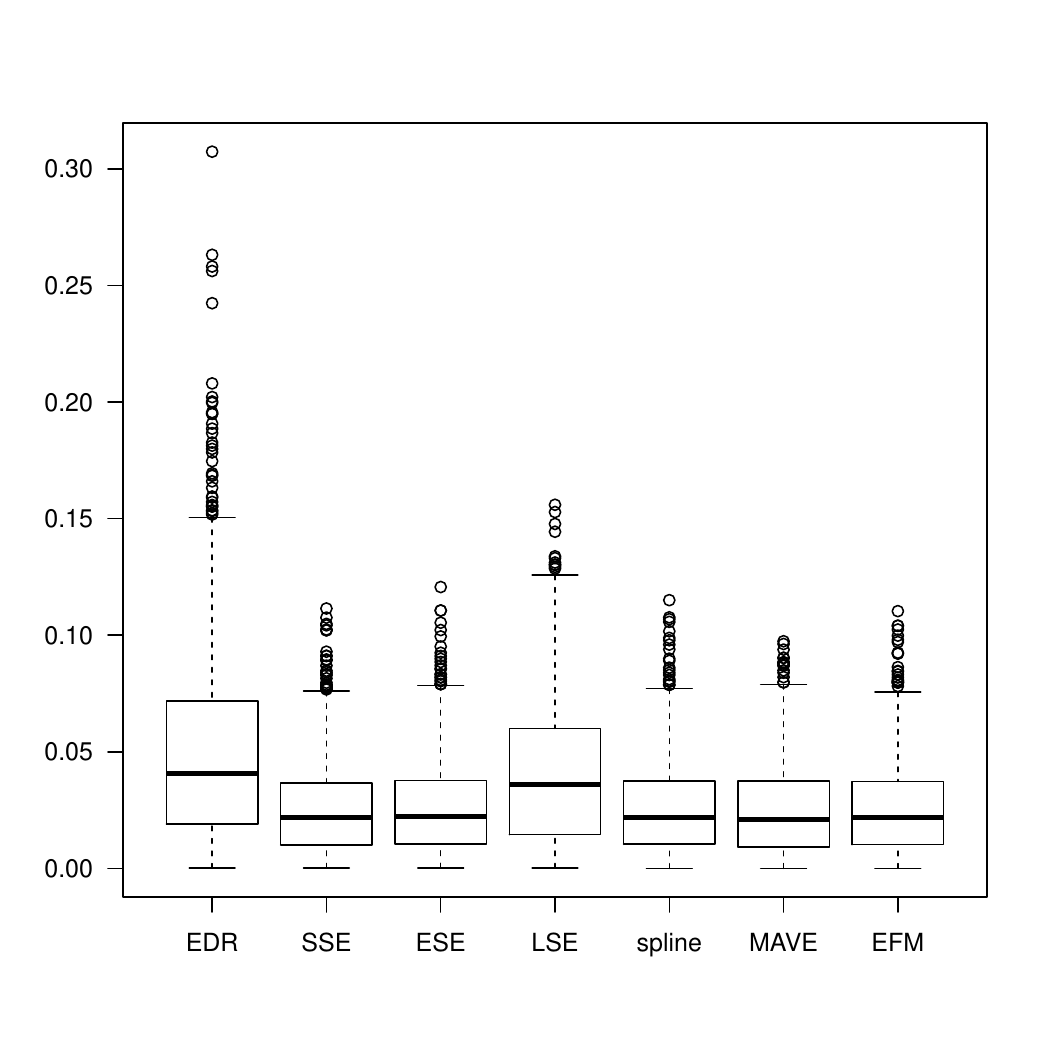}
	\caption{$n = 5000$}
	\end{subfigure}
	\caption{Boxplots of $\sqrt{n/2}\,\|\hat\bma_n-\bma_0\|_2$ for model 2.  In (b) and  (c) the values of EDR were truncated at $0.6$ to show more clearly the differences between the other estimates.}

\end{figure}

\begin{table}
	\label{table:simulation2}
	\caption{Simulation, model 2; $Y_i\sim\text{Bin}\left(10,\exp(\bma_0^T\bmX_i)/\left\{1+\exp(\bma_0^T\bmX_i)\right\}\right)$, where $\bmX_i$ consists of two independent Uniform$(0,1)$ random variables.  The mean value $\hat\mu_i =  \text{mean}(\hat \a_{in}),\, i=1,2$ and $n$ times the variance-covariance $n\text{cov}(\hat\a_{in},\hat\a_{jn}),\,i,\,j=1,2,$ of the Efficient Dimension Reduction Estimate EDR, computed by the {\tt R} package {\tt edr}, the Least Squares Estimate (LSE), the Simple Score Estimate (SSE), the Efficient Score Estimate (ESE), the spline estimate,  the MAVE estimate and the EFM estimate for different sample sizes $n$. The line, preceded by $\infty$, gives the asymptotic values (unknown for EDR and LSE). The values are based on $1000$ replications.}
	\vspace{0.5cm}
	\scalebox{1.0}{
		\begin{tabular}{|lr|cc|ccc |}
			\hline
			Method&$n$ & $\hat\mu_1$ & $\hat\mu_2$  & $\hat \sigma_{11}$& $\hat\sigma_{22}$& $\hat\sigma_{12}$\\
			\hline
			&&&&&&\\
			EDR&100 & 0.587264 & 0.202005  &13.33724&48.15572& 11.87625  \\
			&500 & 0.670702 & 0.602469 & 26.76111 & 66.92737 & 14.09701 \\
			&1000 & 0.696075 & 0.666591  & 21.89080 & 49.31544& 9.345753 \\
			&5000& 0.704424 & 0.706604 & 11.39598 & 11.11493 & -11.17376\\
			\hline
			&$\infty$&0.707107 & 0.707107  & ? & ? & ?\\
		\hline
		&&&&&&\\
			LSE&100 & 0.658631 & 0.699725  &4.069966& 3.596783& -3.609490  \\
			&500 &0.695541 & 0.703007  & 5.650618 & 5.362877 & -5.358190 \\
			&1000 &0.704497 & 0.701243 & 5.909494 & 6.043808 & -5.911246  \\
			&5000& 0.704805 & 0.707621  & 6.303320 & 6.321866 & -6.298515\\
			\hline
			&$\infty$&0.707107 & 0.707107  & ? & ? & ?\\
		\hline
		&&&&&&\\
			SSE&100 & 0.667908 & 0.694376 &3.760921& 3.420387&-3.356968  \\
			&500 & 0.698498 & 0.706423 & 3.358458 & 3.182044 & -3.223734\\
			&1000 & 0.707276 & 0.702390  & 3.179623 & 3.236283 & -3.184724  \\
			&5000& 0.706162 & 0.707286  & 2.718742 & 2.707549 & -2.709870 \\
			\hline
			&$\infty$&0.707107 & 0.707107  & 2.727482 &2.727482 & -2.727482\\
		\hline
		&&&&&&\\
			ESE&100 & 0.684804 & 0.688063  &2.892165& 2.874755& -2.744223  \\
			&500 & 0.698078& 0.706159  & 3.562625 & 3.457337 & -3.446605 \\
			&1000 & 0.707879 & 0.701445  & 3.420159& 3.470217 & -3.418606  \\
			&5000& 0.706321 & 0.707110  & 2.775092 & 2.760287 &-2.764230 \\
			\hline
			&$\infty$&0.707107 & 0.707107  & 2.737200 & 2.737200 & -2.737200\\
		\hline
			&&&&&&\\
			spline&100 &0.677287& 0.695301  &3.009781&2.779876 & -2.714928  \\
			&500 & 0.699117 & 0.706946  & 2.952928 & 2.784383 & -2.830415 \\
			&1000 &0.707890 & 0.702001  & 3.027712 & 3.064772 & -3.026082  \\
			&5000& 0.706200& 0.707312  & 2.764447 & 2.762986 & -2.760530 \\
			\hline
			&$\infty$&0.707107 & 0.707232  & 2.737200 & 2.737200 & -2.737200\\
		\hline
			&&&&&&\\
			MAVE&100 & 0.667849 & 0.654361  &3.891510& 8.700093& -2.325804  \\
			&500 & 0.699108 & 0.706377  & 3.155191 & 2.990569 & -3.031249 \\
			&1000 & 0.707520 & 0.702341  & 3.040201 & 3.097965 & -3.049075  \\
			&5000& 0.707657& 0.705827  & 2.572343 & 2.573418 & -2.570275 \\
			\hline
			&$\infty$&0.707107 & 0.707107  & 2.737200 & 2.737200 & -2.737200\\
		\hline
			&&&&&&\\
			EFM&100 & 0.663227 & 0.666070  &5.681573& 5.978194& -2.503058  \\
			&500 &0.698920 & 0.706295  & 3.279110 & 3.055940 & -3.118757 \\
			&1000 &0.707878 & 0.706275  & 3.102414 & 3.157143 &-3.108516  \\
			&5000& 0.706043&0.701894  & 2.669352 & 2.650343 & -2.656742 \\
			\hline
			&$\infty$&0.707107 & 0.707107  & 2.737200 & 2.737200 & -2.737200\\
		\hline

		\end{tabular}}
	\end{table}

It is clear that the  estimate EDR is inferior to the other methods for these models; even the LSE for which we do not know the rate of convergence has a better performance. In \cite{hristache01} it is assumed that the errors have a normal distribution, but also in model 1, where this condition is satisfied, the behavior is clearly inferior, in particular for the lower sample sizes.

\section{Concluding remarks}
We replaced the ``crossing of zero'' estimators in \cite{FGH:19} by profile least squares estimators. The asymptotic distribution of the estimators was determined and its behavior illustrated by a simulation study, using the same models
 as in \cite{BDJ:19}.
 
 In the first model the error is independent of the covariate and homoscedastic and in this case four of the estimators were efficient. In the other (binomial-logistic) model the error was dependent on the covariates and not homoscedastic. It was shown that the SSE (Simple Score Estimate) had in fact a smaller asymptotic variance in this model than the other estimators for which the asymptotic variance is known, although the difference is very small and does not really show up in the simulations.
 
There is no uniformly best estimate in our simulation, but the EDR estimate is clearly inferior to the other estimates, inluding the LSE, in particular for the lower sample sizes. On the other hand, the LSE is inferior to the other estimators except the EDR. All simulation results can be reproduced by running the {\tt R} scripts in \cite{github:18}.

\section*{Acknowledgement}
\label{section:acknowledgement}
We thank Vladimir Spokoiny for helpful discussions during the Oberwolfach meeting ``Statistics meets Machine Learning'', January 26 - February 1, 2020.        

\bibliographystyle{plainnat}
\bibliography{cupbook}

\end{document}